\documentclass[12pt]{article}
\usepackage[top=1in, bottom=1in, left=1in, right=1in]{geometry}

\usepackage{xcolor}
\usepackage{geometry}
\usepackage{amssymb,bm}
\usepackage{amsmath}
\usepackage{amsthm}
\usepackage{graphicx}
\usepackage{hyperref}
\usepackage{caption}
\usepackage{float} 
\hypersetup{pdfborder=0 0 0}

\newtheorem{theorem}{{\sc Theorem}}[section]

\newtheorem{lemma}[theorem]{{\sc Lemma}}
\newtheorem{corollary}[theorem]{Corollary}
\newtheorem{remark}[theorem]{Remark}

\newtheorem{conjecture}[theorem]{Conjecture}

\newcommand\restr[2]{{ \left.\kern-\nulldelimiterspace #1 \vphantom{\big|} \right|_{#2}}}
\newcommand{\RR}{\mathbb{R}}
\newcommand{\ZZ}{\mathbb{Z}}

\newcommand{\CA}{\mathcal{A}}

\newcommand{\CM}{\mathcal{M}}
\newcommand{\CI}{\mathcal{I}}
\newcommand{\CG}{\mathcal{G}}

\newcommand{\CR}{\mathcal{R}}
\newcommand{\CL}{\mathcal{L}}
\newcommand{\CU}{\mathcal{U}}
\newcommand{\CB}{\mathcal{B}} 
\newcommand{\n}{\noindent}
\newcommand{\comment}[1]{}
\renewcommand{\div}{\mathrm{div}}
\newcommand{\im}{\mathfrak{Im}}
\newcommand{\re}{\Re\mathfrak{e}}
\newcommand{\jump}[1]{\lbrack\!\lbrack #1 \rbrack\!\rbrack}

\title{Far field broadband approximate cloaking for the Helmholtz equation with a Drude-Lorentz refractive index}
\author{Fioralba Cakoni\footnote{Department of Mathematics, Rutgers University, New Brunswick, NJ, USA (fc292@math.rutgers.edu)},  \qquad Narek Hovsepyan\footnote{Department of Mathematics, Rutgers University, New Brunswick, NJ, USA (narek.hovsepyan@rutgers.edu)} \quad and \quad Michael S. Vogelius\footnote{Department of Mathematics, Rutgers University, New Brunswick, NJ, USA (vogelius@math.rutgers.edu)}}
\date{}

\begin{document}
\maketitle

\begin{abstract}
This paper concerns the analysis of a passive, broadband approximate cloaking scheme for the Helmholtz equation in ${\mathbb R}^d$ for $d=2$ or $d=3$. Using ideas from transformation optics, we construct an approximate cloak by ``blowing up" a small ball of radius $\epsilon>0$ to one of radius $1$. In the anisotropic cloaking layer resulting from the ``blow-up" change of variables, we incorporate a Drude-Lorentz-type model for the index of refraction, and we assume that the cloaked object is a soft (perfectly conducting) obstacle. We first show that (for any fixed $\epsilon$) there are no real transmission eigenvalues associated with the inhomogeneity representing the cloak, which implies that the cloaking devices we have created will not yield perfect cloaking at any frequency, even for a single incident time harmonic wave. Secondly, we establish estimates on the scattered field due to an arbitrary time harmonic incident wave. These estimates show that, as $\epsilon$ approaches $0$, the $L^2$-norm of the scattered field outside the cloak, and its far field pattern, approach $0$ uniformly over any bounded band of frequencies. In other words: our scheme leads to broadband approximate cloaking for arbitrary incident time harmonic waves.

\end{abstract}

\section{Introduction}
\setcounter{equation}{0}
In this paper we analyze a passive, broadband approximate cloaking scheme for the Helmholtz equation in ${\mathbb R}^d$ for $d=2$ or $d=3$. Specifically, we are interested in making a bounded region approximately invisible to a far field observer and to  probing by incident fields at arbitrary frequencies, independently of the material inside this region. Using ideas from transformation optics we achieve this by surrounding the region  with a layer of an appropriate  anisotropic material. By including a layer of extremely high conductivity adjacent to the region, we may without loss of generality assume that the region we want to cloak is ``soft", that is, supports a homogeneous Dirichlet boundary condition. The approach of cloaking by mapping, also known as transformation optics, has been popularized by Pendry, Schuring and Smith \cite{PSS06} and Leonhardt \cite{Leon} for  Maxwell's equations. The basic idea is to make a singular change of variables which blows up a point (invisible to any probing incident wave) to a cloaked region. The same  idea had previously been used by Greenleaf, Lassas and Uhlmann to create anisotropic objects that were invisible to EIT \cite{GLU03} (see also \cite{GKLU09}). The singular nature of the perfect cloaks presents various difficulties: in practice this means they are hard to fabricate, and from the analysis point of view in some cases the rigorous definition of the corresponding electromagnetic fields is not obvious \cite{GKLU07, W08-2, W08}. To avoid the use of singular materials in the cloak, regularized schemes have been suggested \cite{KSVW08, KOV10, phys1, phys2}. The trade-off is that such schemes only lead to approximate cloaking. We refer the reader to \cite{ammari1, CsV22, GLU08, GV14} for work on enhancement of approximate cloaks.

To design a passive approximate cloaking device, we blow up a small ball $B_\epsilon$ of radius $\epsilon>0$ (the regularization parameter) to the ball $B_1$ of radius one, which represents the cloaked region. To be more precise we actually map $B_2\setminus B_\epsilon$ onto $B_2\setminus B_1$, keeping fixed the outer boundary $\partial B_2$. $B_2\setminus B_1$ represents the cloak. As result of this change of variables one obtains an anisotropic layer in $B_2\setminus B_1$. We include a Drude-Lorentz-type term (see e.g. \cite{J01}) in the refractive index of the cloaking layer. This results in a frequency dependent and complex valued index of refraction which is consistent with causality. Since the cloaked region $B_1$ is ``soft" we impose a zero Dirichlet boundary condition on the boundary $\partial B_1$. As mentioned earlier, this Dirichlet condition may be viewed as a limit of a highly conducting layer, and it thus may be interpreted as ``hiding" the contents  of $B_1$. A main focus of this  paper is to establish estimates on the scattered field outside the cloak in terms of the small parameter $\epsilon>0$ and the probing  frequencies. We remark that the choice of $B_1$ and $B_2\setminus{B_1}$ for the cloaked region and the cloak, respectively, is made for convenience and one can use more general domains in the change of variables.  We also note that, in the context of approximate cloaking for the Helmholtz equation (the frequency domain wave equation), the Drude-Lorentz model was previously used by Nguyen and Vogelius in \cite{NV16}. The Drude-Lorentz model takes into account the effect of the oscillations of free electrons on the electric permittivity by means of a simple harmonic oscillator model. When viewed in (complex) frequency domain, the refractive index associated with the Drude-Lorentz model may be extended analytically to the whole upper half plane. It is well-known that an immediate consequence of this is causality for the associated non-local time-domain wave equation, see \cite{J01, T56}. This property is most essential for the well-posedness (and the physical relevance) of this equation. Another well known consequence of this analyticity property are the so-called Kramers-Kronig relations between the real and the imaginary part of the refractive index (they are essentially related by Hilbert transforms). However, this fact is not explicitly used in our analysis.

We investigate two questions related to the scattering by the aforementioned cloak $B_2 \setminus B_1$. The first one is  whether, for a fixed $\epsilon>0$, there are wave numbers (proportional to frequencies) and incident fields for which the corresponding scattered field is zero, i.e., the cloak (and $B_1$) is perfectly invisible to this particular probing experiment. This question is related to the existence of real eigenvalues of the interior transmission eigenvalue problem defined on $B_2 \setminus B_1$ \cite{CCH16}, for which that part of the eigenfunction, which corresponds to the incident field, is extendable as a solution to the Helmholtz equation in all of ${\mathbb R}^d$ \cite{CV, CVX}. In particular, such non-scattering wave numbers, for which perfect cloaking is achieved for a particular incident field, form a subset of the real transmission eigenvalues. We prove that, real transmission eigenvalues do not exist for the inhomogeneity presented by the cloak, i.e., for the anisotropic inhomogeneity $B_2\setminus B_1$ with the complex-valued frequency dependent Drude-Lorentz term and a homogeneous Dirichlet condition on the inner boundary $\partial B_1$. In addition, we show that all the (complex) transmission eigenvalues, that lie outside a precisely characterized compact set of the lower half plane, form a countable set with no finite accumulation points outside this compact set. Supported by some computational evidence, we conjecture that a sequence of complex transmission eigenvalues accumulate at a point (as well as at its symmetric counterpart) on the boundary of this compact set. These points have imaginary part equal to $-1/2$, but real parts that depend on the resonant frequency of the Drude-Lorentz term. A complete analysis of the transmission eigenvalue problem for inhomogeneities with such a Drude-Lorentz term is still open. This eigenvalue problem, in addition to being non-selfadjoint, is nonlinear since the  Drude-Lorentz term involves the eigenvalue parameter in a non-linear fashion, and thus the known approaches do not apply \cite{CCH16}. If the Drude-Lorentz term is not present, the existence of an infinite set of  real transmission eigenvalues accumulating at $+\infty$ for (anisotropic) inhomogeneities containing a Dirichlet obstacle is proven in \cite{CCH12, CKW21}. Secondly, although perfect cloaking is impossible at any frequency (even for a single incident wave) we prove that one can achieve approximate cloaking over any given finite band of wave numbers for sufficiently small $\epsilon>0$. In particular, we prove that provided the Drude-Lorentz resonant frequency $k_\epsilon$ is sufficiently large, more precisely $k_\epsilon^2>c_*\epsilon^{-3}$ for $d=3$, and $k_\epsilon^2 > c_* |\ln \epsilon|/\epsilon$ for $d=2$, then for any fixed $R$ the $L^2$-norm of the scattered field in $B_R\setminus B_2$ is of order $\epsilon$ in ${\mathbb R}^3$ and of order $1/|\ln \epsilon|$ in ${\mathbb R}^2$, with a constant depending on the given band of wave numbers, $c_*$ and $R$.
These estimates hold for a large class of incident waves, including plane waves and their superpositions (Herglotz waves). We note that point source waves with sources outside the cloak, as well as their superpositions would also be admissible. Furthermore, we prove that the far field pattern is uniformly $O(\epsilon)$ in ${\mathbb R}^3$ and $O(1/|\ln \epsilon|)$ in ${\mathbb R}^2$, with constants depending on the given band of wave numbers. These latter results are obtained by estimating the norm of the Lippmann-Schwinger  volume integral over $B_2\setminus B_1$ and using scattering estimates adapted from \cite{NV12-1}. We should mention that cloaking via change of variables for the Helmholtz equation at any frequency is investigated in \cite{N10, N12,  NV12, NV12-1, NV16}, but in these papers the region is cloaked to an active source compactly  supported in the exterior of the cloak. The scattering problem with incident field cannot be written in this framework. In fact, in that case the scattered field may be viewed as satisfying an inhomogeneous Helmholtz equation with a source given by the incident field, but this source is supported inside the cloak. Finally let us mention that perfect cloaking for the quasi-static Helmholtz equation (i.e., at zero frequency) with incident plane wave is investigated in \cite{CM17}. One of the results proven there, namely that perfect cloaking is only possible at a discrete set of frequencies is entirely consistent with the fact that we in the present context show that there are no real transmission eigenvalues. Due to the lack of real transmission eigenvalues the lower bounds on cloaking effects provided in \cite{CM17} are not very relevant here. In contrast our analysis demonstrates the possibility of broadband approximate cloaking in a certain (constitutive) regime.

\section{Preliminaries}
\setcounter{equation}{0}

Let $B_r \subset \RR^d$, $d=2$ or $d=3$ denote the open ball of radius $r>0$ centered at the origin and let $S_r = \partial B_r$. For a small parameter $\epsilon>0$ consider the following continuous and piecewise smooth mapping:

\begin{equation} \label{F}
F(x) = 
\begin{cases}
x, \hspace{1in} &x \in \RR^d \setminus B_2
\\[.05in]
\displaystyle \left( \frac{2-2\epsilon}{2-\epsilon} + \frac{|x|}{2-\epsilon} \right) \frac{x}{|x|}, \qquad  &x \in B_2 \setminus B_\epsilon
\end{cases}
\end{equation}

\n For simplicity of notation we will suppress the dependence of $F$ on the parameter $\epsilon$. Note that $F$ maps $B_2\setminus B_\epsilon$ onto $B_2\setminus B_1$, $S_\epsilon$ onto $S_1$,  and that $F(x)=x$ on $S_2$. Now, we design a cloaking device, occupying  $B_2\setminus B_1$, to approximately cloak the (soft) region $B_1$. We incorporate a Drude-Lorentz type term to account for a more physically relevant nonlinear dependence of the index of refraction on wavenumber. The constitutive material properties are thus given by 

\begin{equation} \label{Ac and qc}
A_{c}(x); \ q_{c}(x,k) = 
\begin{cases}
I; \ 1, \hspace{1in} &x \in \RR^d \setminus B_2
\\[.05in]
F_* I; \ F_*1 + \sigma_\epsilon(k), & x \in  B_2 \setminus B_1, 
\end{cases}
\end{equation}

\n where $I$ denotes the $d \times d$ identity matrix and $\sigma_\epsilon$ is the Drude-Lorentz term given  by 

\begin{equation} \label{sigma(k)}
\sigma_\epsilon(k) = \frac{1}{k_\epsilon^2 - k^2 - ik}~,
\end{equation}
\n cf. \cite{J01}, page 331. Here $k_\epsilon > \frac{1}{2}$ represents the so-called resonant frequency  of the Drude-Lorentz model. $F_*$ denotes the push-forward by the map $F$, defined by
\begin{equation*}
F_* A (y) = \frac{DF(x) A(x) DF^T(x)}{|\det DF(x)|}, \qquad F_* q(y) = \frac{q(x)}{|\det DF(x)|}, \qquad x=F^{-1}(y)~,  
\end{equation*}
\n for a matrix-valued function $A$, and for a scalar function $q$, respectively. The definition of the push-forward is motivated by the following change of variables property, which can be proven by straightforward calculations (cf. \cite{GLU03,KSVW08}).

\begin{lemma} \label{LEM change of var} 
Let $F$ be as defined in (\ref{F}). Assume $A \in \left[L^\infty (B_2\setminus \overline{B_\epsilon}) \right]^{d \times d}$ and $q \in L^\infty(B_2\setminus \overline{B_\epsilon})$. Then $u \in H^1(B_2\setminus \overline{B_\epsilon})\cap \{ u=0 \hbox{ on } S_\epsilon \}$ solves the equation

\begin{equation*}
\div(A \nabla u) + q u  = 0, \qquad \text{in} \quad B_2\setminus \overline{B_\epsilon} ~,
\end{equation*}

\n iff $v = u \circ F^{-1} \in H^1(B_2\setminus \overline{B_1})\cap \{ u=0 \hbox{ on } S_1\}$ solves

\begin{equation*}
\div(F_*A \nabla v) + F_*q u  = 0, \qquad \text{in} \quad B_2\setminus \overline{B_1}~.
\end{equation*}

\n The functions $u$ and $v$ satisfy the boundary relations 

\begin{equation}
u = v, \quad \text{and} \quad A \nabla u \cdot \nu = F_*A \nabla v \cdot \nu, \qquad \text{on} \quad S_2~,
\end{equation}

\n where $\nu$ denotes the unit outward normal vector on $S_2$ and the equality of the conormal derivatives is understood in the sense of distributions in $H^{-\frac{1}{2}}(S_2)$.

\n Furthermore\footnote{Since similar formulas hold for $F_*\left[(F^{-1})_* B\right]$ and $F_*\left[(F^{-1})_* p\right]$ it follows that $(F^{-1})_*= (F_*)^{-1}$, and for that reason we sometimes use the notation $F_*^{-1}$ for both.}
$$
(F^{-1})_*\left[ F_* A \right]=A~,~\hbox{ and }~ (F^{-1})_*\left[ F_* q \right]=q~. 
$$
\end{lemma}

\n Let $u^i$ be an incident field at a given wave number $k>0$ (we suppress the dependence of $u^i$ on $k$ for the ease of notation), i.e.,

\begin{equation}
\Delta u^i + k^2 u^i = 0, \qquad \text{in} \quad \RR^d.
\end{equation}

\n Given the incident wave $u^i$ and the ``cloaked" soft obstacle $B_1$, consider now the associated Helmholtz scattering problem. If $A_c$ and $q_c$ denote the constitutive material properties defined in (\ref{Ac and qc}), then the total field $u_c \in H^1_{loc}(\RR^d\setminus \overline{B}_1)$ is the unique solution to

\begin{equation} \label{u_c pde}
\begin{cases}
\div ( A_c \nabla u_c ) + k^2 q_c u_c = 0, \qquad \qquad &\text{in} \quad \RR^d \setminus \overline{B}_1,
\\
u_c = 0, & \text{on} \quad S_1~,
\end{cases}
\end{equation}

\n of the form

\begin{equation} \label{u_c def sc tr inc}
u_c =
\begin{cases}
u_c^t, \qquad \qquad &\text{in} \quad B_2\setminus \overline{B}_1~,
\\
u^i + u_c^s & \text{in} \quad \RR^d \setminus \overline{B}_2~,
\end{cases}
\end{equation}

\n where $u_c^t$ is the transmitted field and $u_c^s$ is the scattered field, which satisfies the Sommerfeld radiation condition

\begin{equation} \label{Sommerfeld cond}
\lim_{r \to \infty} r^{\frac{d-1}{2}} \left( \partial_r u_c^s - ik u_c^s \right) = 0, \quad \text{as} \quad r=|x| \to \infty~,
\end{equation}

\n uniformly in $\hat{x} = x / |x|$ (cf. \cite{CK19} for more details about the scattering problem).  As $u_c$ and its conormal derivative are continuous across $S_2$, the problem \eqref{u_c pde} can equivalently be written 

\begin{equation} \label{main u_c}
\begin{cases}
\Delta u_c^s + k^2 u_c^s = 0, \hspace{1in} &\text{in} \quad \RR^d \setminus \overline{B}_2
\\[.05in]
u_c^s \ \text{satisfies the outgoing radiation condition}
\\[.05in]
\nabla \cdot \left( A_c \nabla u_c^t \right) + k^2 q_c u_c^t = 0, &\text{in} \quad B_2 \setminus \overline{B}_1
\\[.05in]
\Delta u^i + k^2 u^i = 0, &\text{in} \quad \RR^d
\\[.05in]
u_c^t = u^i + u_c^s,  &\text{on} \quad S_2
\\[.05in]
A_c \nabla u_c^t \cdot \nu = \partial_{\nu} u^i + \partial_{\nu} u_c^s,  &\text{on} \quad S_2
\\[.05in]
u_c^t = 0, &\text{on} \quad S_1.
\end{cases}
\end{equation}

\n As the scattered field $u_c^s$ satisfies the constant coefficient Helmholtz equation, it is in fact real analytic and admits the following asymptotic behavior as $r \to \infty$:

\begin{equation} \label{far field}
u_c^s(x) = \frac{e^{ikr}}{r^{\frac{d-1}{2}}} u^\infty(\hat{x}) + O \left( r^{-\frac{d+1}{2}} \right),
\end{equation}

\n where the function $u^\infty$, defined on $S_1$, is the so-called far field pattern of the scattered field $u_c^s$. It is well-known that the vanishing of $u^\infty$ on $S_1$, implies the vanishing of the scattered field $u_c^s$ in $\RR^d \setminus \overline{B}_2$ (cf. Rellich's Lemma in \cite{CK19}). A non-trivial incident field $u^i$ and the wave number $k>0$ for which the corresponding far field pattern vanishes are referred to as non-scattering incident field and a non-scattering wave number, respectively. If we regard $u^i$ as a function defined in $B_2$, then from \eqref{main u_c} it is clear that at a non-scattering wave number $k>0$, there exist non-trivial functions $w_c = u_c^t$ and $v=u^i$ defined in $B_2 \setminus \overline{B}_1$ and $B_2$, respectively, such that

\begin{equation} \label{main transmission}
\begin{cases}
\nabla \cdot \left( A_c \nabla w_c \right) + k^2 q_c w_c = 0, &\text{in} \quad B_2 \setminus \overline{B}_1
\\[.05in]
\Delta v + k^2 v = 0, &\text{in} \quad B_2
\\[.05in]
w_c = v,  &\text{on} \quad S_2
\\[.05in]
A_c \nabla w_c \cdot \nu = \partial_{\nu} v,  &\text{on} \quad S_2
\\[.05in]
w_c = 0, &\text{on} \quad S_1.
\end{cases}
\end{equation}

\n A wave number $k$ for which \eqref{main transmission} admits a non-trivial solution is called an interior transmission eigenvalue with the corresponding eigenfunction $(w_c, v)$. Thus, non-scattering wave numbers are necessarily real interior transmission eigenvalues \cite{CCH16}. Conversely, a real interior transmission eigenvalue $k>0$ is a non-scattering wave number if the eigenvector $v$ can be extended from $B_2$ to a solution of the Helmholtz equation in all of $\RR^d$ \cite{CVX, CV}.

\section{Main Results}
\setcounter{equation}{0}
For clarity and the reader's convenience we now state the main results of our paper.  The first theorem addresses the question whether our cloak provides a perfect cloaking of the region $B_1$ for even a single incident wave.  
\begin{theorem} \label{THM transmission} \mbox{}
Consider the interior transmission eigenvalue problem \eqref{main transmission}.
\begin{enumerate}
\item[(i)] There are no interior transmission eigenvalues  in $\RR \cup i\RR$.

\item[(ii)] $k \in \mathbb{C}$ is an interior transmission eigenvalue if and only if so is $-\overline{k}$.

\item[(iii)] Assume $k_\epsilon > \frac{1}{\sqrt{2}}$, let $\kappa = \sqrt{k_\epsilon^2 - \frac{1}{4}} - \frac{i}{2}$ and let $K$ be the shaded compact region in Figure~\ref{FIG evals}. The region $K$ is symmetric about the imaginary axis, the slanted line segment of the boundary in the right half-plane has the equation $\im k = - \re k$, the curved arc joining $\kappa$ to $k_\epsilon$ is given by $\re k = \sqrt{(\im k)^2 + \im k + k_\epsilon^2}$. Let $\CG$ denote the open set $\CG = \mathbb{C} \setminus K$. Then those interior transmission eigenvalues which lie inside $\CG$ form a discrete set (i.e., an at most countable set with no limit points in $\CG$).
\end{enumerate}
\end{theorem}

\begin{figure}[h]
\center
\captionsetup{width=.7\linewidth}
\includegraphics[scale=0.6]{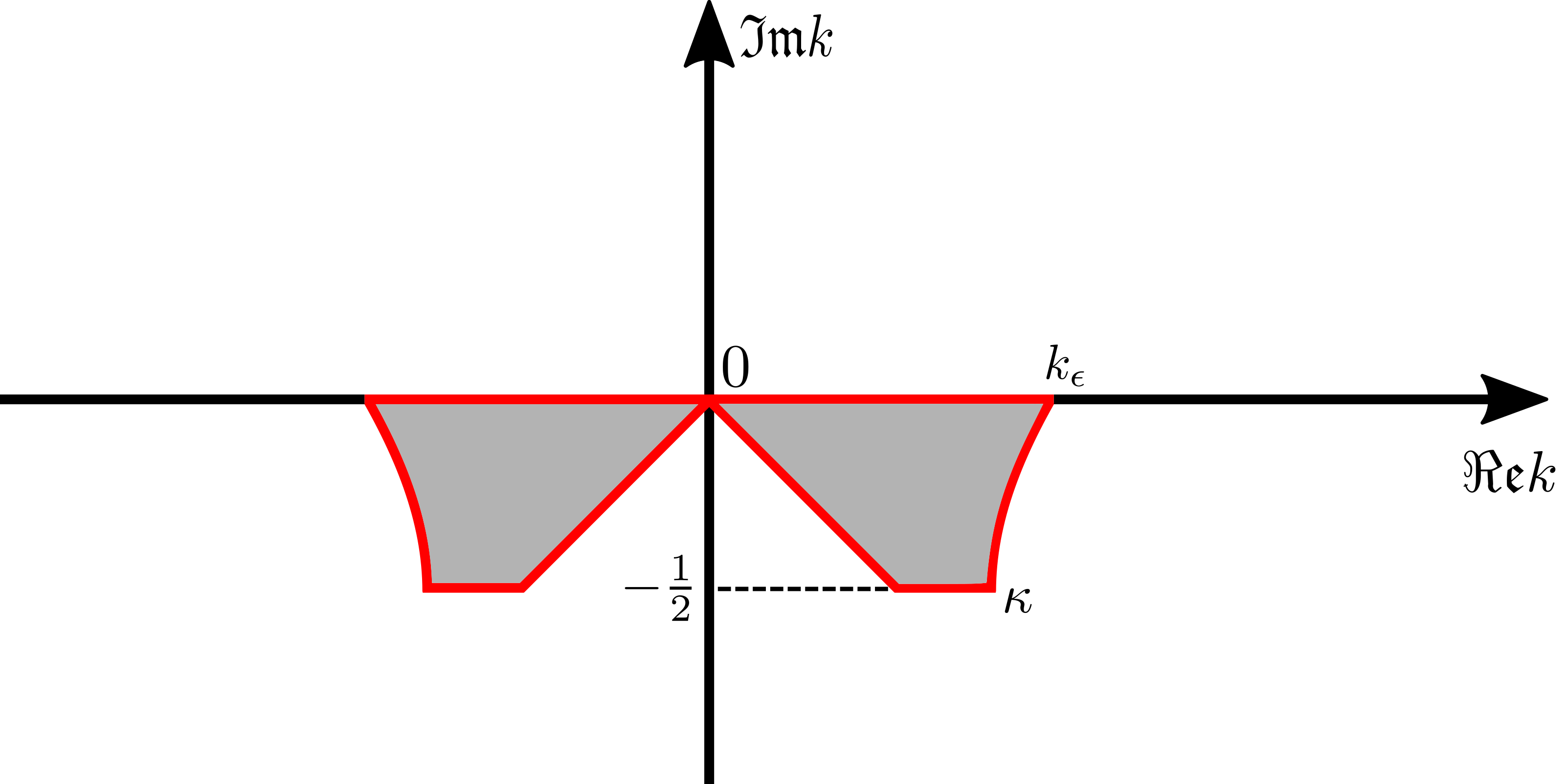}
\caption{The shaded compact region $K$, outside of which the interior transmission eigenvalues of \eqref{main transmission} form a discrete set.}
\label{FIG evals}
\end{figure}
\noindent
Part $(i)$ of Theorem~\ref{THM transmission} will be proven in Section~\ref{SECT trans eval}. As a consequence we conclude that perfect cloaking/non-scattering is impossible at any wave number $k>0$, since real transmission eigenvalues do not exist. Part $(ii)$ is an immediate consequence of the symmetry relation

$$\overline{\sigma_\epsilon(k)} = \sigma_\epsilon(-\overline{k})~, \qquad \forall \ k \in \mathbb{C}~.$$

\n As a result $q_c(x,k)$ has the same symmetry property and $k$ is a transmission eigenvalue of \eqref{main transmission} with eigenfunction $(w_c,v)$, if and only if so is $-\overline{k}$ with eigenfunction $(\overline{w_c},\overline{v})$. The proof of part $(iii)$ will be given in the Appendix since the discreteness of complex eigenvalues is not central to the cloaking discussion. The value $\kappa$ is one of the poles of $\sigma_\epsilon(k)$ (the other one is $-\overline{\kappa}$). Numerical evidence, presented in Section~\ref{SECT non discret}, indicates that it is a limit point for the set of transmission eigenvalues of \eqref{main transmission}. Being bold, we venture

\begin{conjecture} \label{CONJ} (Finite accumulation point of transmission eigenvalues) \mbox{}

\n Let $\kappa$ be defined as in part $(iii)$ of Theorem~\ref{THM transmission}. Then $\kappa$ is a limit point of transmission eigenvalues of \eqref{main transmission}.
\end{conjecture}
\vskip 5pt

\n We note that Theorem \ref{THM transmission} asserts nothing about potential interior transmission eigenvalues in the set $K \setminus \RR$. Their nature is a completely open problem.

Although perfect cloaking is impossible, we demonstrate that, under a suitable growth assumption on $k_\epsilon$, one can achieve approximate cloaking over any given finite band of wave numbers. We first state the main estimate on the scattered field including its explicit dependence on $k$ (and $\epsilon$). The broadband cloaking estimates follow as a corollary from this. We define
\begin{equation} \label{M}
M_{\epsilon, k} = \|F_*^{-1}q_c - 1\|_{L^\infty(B_2 \setminus B_\epsilon)}= \|F_*^{-1}\sigma_{\epsilon}(k)\|_{L^\infty(B_2 \setminus B_\epsilon)}~,
\end{equation}

\n where $F_*^{-1}$ denotes the push-forward by the map $F^{-1}$, and we set

\begin{equation} \label{a(k)}
a(k)=
\begin{cases}
1, \hspace{1.3in} & d=3~,
\\
\min\{1+|\ln k|, k^{-\frac{1}{4}}\}, & d=2~.
\end{cases}
\end{equation}

\begin{theorem} \label{THM main}
Let $R>2$ and $k_0>0$. Suppose $0<\epsilon k <k_0$ and suppose

\begin{equation} \label{u^i assumption 1}
\|u^i\|_{L^\infty(B_\epsilon)} + \epsilon \|\nabla u^i\|_{L^\infty(B_\epsilon)} \leq C~.
\end{equation}

\n Let $u^s_c$ be the scattered field from \eqref{main u_c}. There exists a constant $c=c(k_0, R)>0$ such that, if $k^2 a(k) M_{\epsilon, k} < c$ then

\begin{equation} \label{u_c^s 3d}
\|u_c^s\|_{L^2(B_R \setminus B_2)} \lesssim \epsilon + k^2a(k) M_{\epsilon, k} \|u^i\|_{L^2(B_R)}~, \qquad \qquad for \ d=3~,
\end{equation}

\n and

\begin{equation} \label{u_c^s 2d}
\|u_c^s\|_{L^2(B_R \setminus B_2)} \lesssim \frac{|H_0^{(1)}(k)|}{|H_0^{(1)}(\epsilon k)|} + k^2a(k) M_{\epsilon, k} \left( 1 +\|u^i\|_{L^2(B_R)} \right)~, \qquad \qquad for \ d=2~,
\end{equation}

\n where the implicit constants in \eqref{u_c^s 3d} and \eqref{u_c^s 2d} depend only on $R, k_0$ and $C$.
\end{theorem}

\begin{remark}
\normalfont In the above theorem, $H_0^{(1)}$ denotes the Hankel function of the first kind of order $0$. We also adopt the following notation: for two positive quantities $A$ and $B$, we write $A \lesssim B$, if there exists a constant $d>0$ (independent of $A$ and $B$) such that $A \leq d B$.
\end{remark}

\n Imposing a suitable lower bound on the resonant frequency $k_\epsilon$ with respect to $\epsilon$, the quantity $M_{\epsilon, k}$ (for bounded $k$) becomes of order $\epsilon$ for $d=3$, and of order $1/|\ln \epsilon|$ for $d=2$ (cf. \eqref{M order eps}) and Theorem~\ref{THM main} implies the following result:

\begin{theorem}(Broadband approximate cloaking) \label{THM main coro} \mbox{}

\n Let $R>2$, $k_+>k_->0$, and set $\Gamma:= [k_-,k_+]$. Assume that for some constant $c_*>0$, \ $k_\epsilon^2 > c_* \epsilon^{-3}$ for $d=3$, and $k_\epsilon^2 > c_* |\ln \epsilon|/\epsilon$ for $d=2$. Furthermore, assume that the incident field $u^i$ satisfies

\begin{equation} \label{u^i assumption 2}
\|u^i\|_{L^2(B_R)} \leq C_R,  \qquad \qquad \ \forall \ k \in \Gamma.
\end{equation}

\n Let $u^s_c$ be the scattered field from\eqref{main u_c}. There exists a constant $c_1=c_1(k_-,k_+,R, c_*)>0$ such that, for all $\epsilon < c_1$ and $k \in \Gamma$

\begin{equation}
\|u^s_c\|_{L^2(B_R \setminus B_2)} \lesssim
\begin{cases}
\epsilon~, \hspace{1.3in} & d=3~,
\\[.01in]
1/ |\ln \epsilon|~, & d=2~,
\end{cases}
\end{equation}

\n where the implicit constant depends only on $k_-,k_+,R, c_*$ and $C_R$. 
Similarly, there exists a constant $c_2=c_2(k_-,k_+,c_*)>0$, such that for all $\epsilon < c_2$, \ $k \in \Gamma$, and $|\hat{x}|=1$ 

\begin{equation}
|u_\infty(\hat{x})| \lesssim
\begin{cases}
\epsilon~, \hspace{1.3in} & d=3~,
\\[.01in]
1/ |\ln \epsilon|~, & d=2~.
\end{cases}
\end{equation}

\n where $u_\infty$ is the far field pattern defined in \eqref{far field}, and the implicit constant depends only on $k_-,k_+, c_*$ and $C_5$.
\end{theorem}

\begin{remark} \mbox{}
\normalfont

\begin{enumerate}
\item[(i)] The results of the above two theorems do not use the radial geometry in any essential way and carry over to the non-radial setting as well.
    
\item[(ii)] The assumption \eqref{u^i assumption 2} (or \eqref{u^i assumption 1}) is satisfied by incident plane waves as well as by their superpositions, the so-called Herglotz waves $u^i:=u_g$  given by 
$$u_g(x)=\int_{|\hat y|=1}g(\hat y)e^{ikx\cdot \hat y}\, ds_{\hat y}~, \qquad g\in L^2(S_1)~.$$ It is also satisfied by radiating point sources (outside of $B_2$) and their appropriate superpositions.
\end{enumerate}
\end{remark}

\section{Transmission eigenvalues} \label{SECT trans eval}
\setcounter{equation}{0}

In this section we study the interior transmission eigenvalue problem. We first eliminate the anisotropy $A_c$  in the formulation \eqref{main transmission} by  using  a change of variables to arrive at  a new interior transmission eigenvalue problem, which has the same eigenvalues as \eqref{main transmission}. Then we reformulate the resulting problem in terms of a fourth order PDE, following \cite{CCH12} (see also \cite{CCH16}). Using this new formulation we prove part $(i)$ of Theorem~\ref{THM transmission}. Furthermore, in Section~\ref{SECT non discret} we present numerical evidence supporting Conjecture~\ref{CONJ} in two dimension. 

\subsection{The variational formulation}

In the interior transmission eigenvalue problem \eqref{main transmission} let us change the variables in $w_c$, while leaving $v$ unchanged. Namely, let

$$w = w_c \circ F,$$

\n where $F$ is defined by \eqref{F}. Using the properties of the map $F$ (namely that $F(x)=x$ on $S_2$, $F$ maps $S_\epsilon$ onto $S_1$ and $F_*^{-1}A_c = F_*^{-1}F_*I=I$ in $B_2 \setminus \overline{B}_1$) along with Lemma~\ref{LEM change of var}, we obtain that $w, v$ solve the following transmission problem:

\begin{equation} \label{trans change of var}
\begin{cases}
\Delta w  + k^2 q w = 0~, \qquad \qquad &\text{in} \ B_2 \setminus \overline{B}_\epsilon
\\
\Delta v + k^2 v = 0~, &\text{in} \ B_2
\\
w = v~,  &\text{on} \ S_2
\\
\partial_\nu w = \partial_\nu v &\text{on} \ S_2
\\
w = 0~, &\text{on} \ S_\epsilon
\end{cases}
\end{equation}

\n where 

\begin{equation*}
q(x,k) = F^{-1}_* q_c (x,k) = F^{-1}_* \left[F_* 1 + \sigma_\epsilon(k) \right] =  1 + \sigma_\epsilon(k) | \det DF(x)|~,  \qquad \qquad x \in  B_2 \setminus \overline{B}_\epsilon~. 
\end{equation*}

\n Let us introduce the notation

\begin{equation*}
{\mathcal O}:=B_2 \setminus \overline{B}_\epsilon. 
\end{equation*}

\n It is clear that $k \in \mathbb{C}$ is a  transmission eigenvalue for \eqref{main transmission} with eigenfunction $(w_c, v)$, if and only if, it is a transmission eigenvalue for \eqref{trans change of var} with eigenfunction $(w=w_c \circ F, v)$. Thus \eqref{main transmission} and \eqref{trans change of var} have the same set of transmission eigenvalues. We recall that the weak solution of \eqref{trans change of var} is a pair of functions\footnote{We use the notation $L^2_\Delta({\mathcal O})= \{ w\in L^2({\mathcal O})~:~ \Delta w \in L^2({\mathcal O})\}$ and $H^1_\Delta({\mathcal O})= \{ w\in H^1({\mathcal O})~:~ \Delta w \in L^2({\mathcal O})\}$} $w \in L^2_\Delta({\mathcal O})$ and $v \in L^2_\Delta (B_2)$ that satisfy the PDEs of \eqref{trans change of var} in the sense of distributions, such that $w =0$ on $S_\epsilon$ and $u:=w-v \in H^1_\Delta({\mathcal O})$ satisfies the boundary conditions $u=\partial_\nu u = 0$ on $S_2$.

\begin{remark} \mbox{}
\normalfont
\begin{enumerate}
\item[(i)] We note that the trace (on $S_\epsilon$) of a function $w\in L^2_\Delta({\mathcal O})$ makes sense as an element of $H^{-\frac{1}{2}}(S_\epsilon)$ by duality, using the identity 
$$\left<w,\tau\right>_{H^{-1/2},H^{1/2}}=\int_{{\mathcal O}}(w\Delta \varphi-\varphi \Delta w)\,dx~,$$
where $\varphi\in H^2({\mathcal O})$ is such that $\varphi =0$ in a neighborhood of $S_2$, and $\varphi=0$ and $\partial \varphi/\partial \nu=\tau$ on $S_\epsilon$ .

\item[(ii)] Similarly we note that for a function $u \in H^1_\Delta ({\mathcal O})$ the normal derivative $\partial_\nu u$ (on $S_2$) makes sense as an element of $H^{-\frac12}(S_2)$ by duality, using the formula
$$
\left<\partial_\nu u,\psi \right>_{H^{-1/2},H^{1/2}}=\int_{{\mathcal O}}\left(\Delta u \varphi + \nabla u \nabla \varphi \right)\,dx~,
$$
where $\varphi\in H^1({\mathcal O})$ is such that $\varphi =0$ on $S_\epsilon$, and $\varphi=\psi$ on $S_2$.
\end{enumerate}
\end{remark}
\n
We can reformulate \eqref{trans change of var} as a fourth order problem. Indeed, given a weak solution $w,v$ of \eqref{trans change of var}, let us set

\begin{equation} \label{u trans}
u = 
\begin{cases}
w-v, \qquad &\text{in} \ {\mathcal O}
\\
-v, & \text{in} \ B_\epsilon,
\end{cases}
\end{equation}

\n It is clear that

\begin{equation} \label{u v PDE}
\Delta u + k^2 q u = k^2 (1-q) v, \qquad \qquad \text{in} \ {\mathcal O}.
\end{equation}

\n Dividing both sides of the above equation by $1-q$ (note that $1-q=-\sigma_\epsilon(k)| \det DF |\neq 0$ in ${\mathcal O}$) and applying the operator $\Delta + k^2$ we can eliminate $v$ and obtain a fourth order equation for $u$. The boundary condition on $w$ implies that $u$ is continuous across $S_\epsilon$. Next, since $v$ solves the Helmholtz equation in $B_2$, $v$ and its normal derivative $\partial_\nu v$ are continuous across $S_\epsilon$. We can rewrite these continuity conditions in terms of $u$ using \eqref{u trans} and \eqref{u v PDE}. Thus, we obtain that $u$ (weakly) solves the problem  

\begin{equation} \label{u PDE trans}
\begin{cases}
\displaystyle (\Delta + k^2) \frac{1}{1-q} (\Delta + k^2 q) u = 0, \qquad \qquad &\text{in} \ {\mathcal O}
\\[.09in]
\Delta u + k^2 u = 0, &\text{in} \ B_\epsilon
\\
u = \partial_\nu u = 0,  &\text{on} \ S_2
\\
u^+ = u^-, &\text{on} \ S_\epsilon
\\
\displaystyle \left[ \frac{1}{1-q} (\Delta + k^2 q) u \right]^+ = -k^2u^-, &\text{on} \ S_\epsilon
\\[.15in]
\displaystyle \partial_\nu^+ \left[\frac{1}{1-q} (\Delta + k^2 q) u \right] = -k^2\partial_\nu^- u, &\text{on} \ S_\epsilon.
\end{cases}
\end{equation}

\n Note that as $v \in L^2(B_2)$ solves the Helmholtz equation, by local elliptic regularity $v \in H^1(B_\epsilon)$. But as $u$ is continuous across $S_\epsilon$, we conclude that $u \in H^1(B_2) \cap H^1_\Delta({\mathcal O})$. Incorporating the boundary conditions on $S_2$ we introduce the Hilbert space of functions

\begin{equation}\label{X}
X = \left\{ u \in H^1(B_2): \Delta u \in L^2({\mathcal O}) \ \text{and} \ u = \partial_\nu u = 0 \ \text{on} \ S_2 \right\},
\end{equation}

\n where $\partial_\nu u \in H^{-\frac{1}{2}}(S_2)$, and is defined as described in the earlier remark. Thus, given a non-trivial weak solution $w,v$ of \eqref{trans change of var}, the function $u \in X$, given by \eqref{u trans}, is a non-trivial weak solution of \eqref{u PDE trans}. Conversely, if $u \in X$ is a non-trivial weak solution of \eqref{u PDE trans}, then 

\begin{equation} \label{w and v from u}
v = 
\begin{cases}
\displaystyle \frac{1}{1-q} \left( \Delta + k^2 q \right) u, \quad &\text{in} \ {\mathcal O}
\\
-k^2u, &\text{in} \ B_\epsilon
\end{cases}
\qquad \text{and} \qquad
w = k^2 u+v,  \quad \text{in} \ {\mathcal O},
\end{equation}

\n satisfy $w \in L^2({\mathcal O}), \ v \in L^2(B_2)$ and $w-v \in H^1_\Delta({\mathcal O})$ and yield a non-trivial weak solution of \eqref{trans change of var}. Integration by parts easily yields a variational formulation of \eqref{u PDE trans}, namely : find $u \in  X$ such that

\begin{eqnarray} \label{variational form}
&&\int_{{\mathcal O}} \frac{1}{1-q} \left( \Delta u + k^2u \right) \left( \Delta \overline{\varphi} + k^2 \overline{\varphi} \right) dx \\&&\hskip 60pt - k^4 \int_{B_2} u \overline{\varphi} dx + k^2  \int_{B_2} \nabla u \cdot \nabla \overline{\varphi} dx= 0,
\qquad \qquad \forall \varphi \in X. \nonumber
\end{eqnarray}
\n
Before excluding the existence of real and purely imaginary transmission eigenvalues we need the following formulas for the map $F$:

\begin{lemma} \label{LEM F}
Let $F$ be given by \eqref{F}, and set $\hat{x} = x/|x|$, then

\begin{equation*}
D F(x) = 
\begin{cases}
I, \hspace{2.5in} & \text{in} \ \RR^d \setminus B_2
\\[.05in]
\dfrac{1}{2-\epsilon} \left\{ I + \frac{2-2\epsilon}{|x|} \left( I - \hat{x} \otimes \hat{x} \right) \right\},   & \text{in} \ B_2 \setminus B_\epsilon
\\[.1in]
I/\epsilon,  & \text{in} \ B_\epsilon,
\end{cases}
\end{equation*}

\n where $I$ is the $d\times d$ identity matrix and for any two vectors $a, b \in \RR^d$, $a \otimes b$ denotes the matrix whose $(i,j)$-th element is $a_i b_j$. In particular,

\begin{equation} \label{det DF}
\det D F(x) = 
\begin{cases}
1 \hspace{1.5in} & \text{in} \ \RR^d \setminus B_2
\\[.05in]
\displaystyle \frac{(2-2\epsilon + |x|)^{d-1}}{(2-\epsilon)^d |x|^{d-1}}   & \text{in} \ B_2 \setminus B_\epsilon
\\[.1in]
1/ \epsilon^d  & \text{in} \ B_\epsilon.
\end{cases}
\end{equation}
\end{lemma}

\begin{proof}
The formulas for $DF(x)$ inside $B_\epsilon$ and outside of $B_2$ are trivial. In the region $B_2\setminus B_\epsilon$ it is a direct consequence of the identity

\begin{equation*}
D \hat{x} = \frac{1}{|x|} \left( I - \hat{x} \otimes \hat{x} \right)
\end{equation*}

\n Finally, using the identity $\det (I + a \otimes b) = 1 + a \cdot b$ for any two vectors $a, b \in \RR^d$, we find that for $x \in B_2 \setminus B_\epsilon$ 

\begin{equation*}
\det DF(x) = \frac{1}{(2-\epsilon)^d} \left[ \frac{2-2\epsilon}{|x|}+1 \right]^{d-1},
\end{equation*}

\n which concludes the proof.
\end{proof}

\begin{lemma} \label{LEM not trans eval}
There are no non-trivial solutions to \eqref{main transmission} for  $k\in \RR \cup i\RR$, {\it i.e.}, there are no transmission eigenvalues for \eqref{main transmission} in $\RR \cup i\RR$.
\end{lemma}

\begin{proof}
First suppose $k=i\tau$ with $\tau \in \RR$ is a transmission eigenvalue. The above discussion shows that the problem \eqref{u PDE trans} has a non-trivial solution $u \in X$ for this value of $k$. Using the variational formulation \eqref{variational form} with $\varphi = u$ we get

\begin{equation} \label{var form u i tau}
0 = \int_{{\mathcal O}} \frac{1}{q-1} \left| \Delta u - \tau^2 u \right|^2 + \tau^4 \int_{B_2}  |u|^2 dx +  \tau^2 \int_{B_2} |\nabla u|^2 dx~.
\end{equation}

\n Note that

\begin{equation*}
q(x,i\tau) - 1 = \sigma_\epsilon(i \tau)| \det DF(x)| = \frac{1}{k_\epsilon^2+\tau^2 + \tau} \frac{(2-2\epsilon + |x|)^{d-1}}{(2-\epsilon)^d |x|^{d-1}}~, \qquad \qquad x \in {\mathcal O}~. 
\end{equation*}

\n If $\tau\ge 0$ the above quantity is obviously positive. For $\tau<0$, it is still positive due to the assumption $2k_\epsilon>1$. Thus $q(x,i\tau) - 1 > 0$ for all $\tau \in \RR$ and $x \in {\mathcal O}$. For $\tau \ne 0$ we now conclude from \eqref{var form u i tau} that $u=0$ in $B_2$, contradicting the non-triviality of $u$ for $\tau \ne 0$. For $\tau =0$ we conclude from \eqref{var form u i tau} that $\Delta u=0$ in $\mathcal O$. The Cauchy boundary conditions on $S_2$ now imply that $u=0$ in $\mathcal O$, and the continuity of $u$ across $S_\epsilon$ in combination with the fact that $\Delta u=0$ in $B_\epsilon$ yields that $u=0$ in all of $B_2$, contradicting the non-triviality of $u$ also for $\tau =0$.

\n
Assume now that $k \in \RR \setminus \{0\}$ is a transmission eigenvalue; again let $\varphi = u$ in the variational formulation \eqref{variational form} and take the imaginary part of the resulting equation to conclude that

\begin{equation*}
0 = \int_{{\mathcal O}} \im \left(\frac{1}{q-1} \right) \left| \Delta u + k^2 u \right|^2 dx = \frac{k}{|k_\epsilon^2-k^2-ik|^2} \int_{{\mathcal O}} \frac{(2-2\epsilon + |x|)^{d-1}}{(2-\epsilon)^d |x|^{d-1}} \left| \Delta u + k^2 u \right|^2 dx~.
\end{equation*}

\n Therefore $\Delta u + k^2 u = 0$ in ${\mathcal O}$. Using the boundary conditions $u = \partial_\nu u = 0$ on $S_2$, we conclude that $u = 0$ in ${\mathcal O}$. Since $k\ne 0$ also conclude from the boundary conditions of \eqref{u PDE trans} that $u^-=\partial_\nu^- u =0$ on $S_\epsilon$. The fact that $\Delta u + k^2 u = 0$ in $B_\epsilon$ now implies that $u=0$ in $B_\epsilon$, and thus $u=0$ in all of $B^2$. This contradicts the non-triviality of $u$.
\end{proof}

\subsection{Numerical evidence of finite accumulation points  of transmission eigenvalues} \label{SECT non discret}

\n In this section we assume that $d=2$ and consider the transmission eigenvalue problem after change of variables, i.e., the problem \eqref{trans change of var}. In polar coordinates $(r, \theta)$ we can expand the functions $v$ and $w$ as follows:

\begin{equation} \label{v in polar}
v(r,\theta) = \sum_{n \in \ZZ} \gamma_n J_n \left( k r \right) e^{in\theta},
\qquad \qquad
w(r, \theta) = \sum_{n \in \ZZ} \left[ \alpha_n {\mathcal A}_n(r) + \beta_n {\mathcal B}_n(r) \right] e^{in\theta}
\end{equation}

\n where $\alpha_n, \beta_n, \gamma_n$ are complex constants, $J_n$ is the Bessel function of order $n$ and ${\mathcal A}_n, {\mathcal B}_n$ (which also depend on $k$ and $\epsilon$) are linearly independent solutions of 

\begin{equation*}
r^2 R'' + r R' + \left[ k^2 r^2 + k^2 \sigma_\epsilon(k) \frac{r(r+2-2\epsilon)}{(2-\epsilon)^2} - n^2 \right] R = 0~.
\end{equation*}

\n The boundary conditions of \eqref{trans change of var} can be rewritten as

\begin{equation*}
\begin{cases}
\alpha_n {\mathcal A}_n(2) + \beta_n {\mathcal B}_n(2) &= \gamma_n J_n(2k)
\\[.05in]
\alpha_n A_n'(2) + \beta_n {\mathcal B}_n'(2) &= \gamma_n k J_n'(2k)
\\[.05in]
\alpha_n {\mathcal A}_n(\epsilon) + \beta_n {\mathcal B}_n(\epsilon) &= 0~.
\end{cases}
\end{equation*}

\n To obtain a nontrivial solution $(v,w)$ (i.e., to ensure that $k$ is an interior transmission eigenvalue) we need that there exists some $n \in \ZZ$ such that

\begin{equation*}
f(n,k):=\det \CM = 0~,
\end{equation*}

\n where

\begin{equation*}
\CM = 
\begin{pmatrix}
{\mathcal A}_n(2) & {\mathcal B}_n(2) & -J_n(2k)
\\
A_n'(2) & {\mathcal B}_n'(2) & - kJ_n'(2k)
\\
{\mathcal A}_n(\epsilon) & {\mathcal B}_n(\epsilon) & 0
\end{pmatrix}~.
\end{equation*}

\n The functions $A_n, {\mathcal B}_n$ can be expressed in terms of the Whittaker functions as follows:

\begin{equation} \label{A_n Whittaker M}
{\mathcal A}_n(r) = \frac{1}{\sqrt{r}} M_{\lambda_\epsilon(k), |n|}\left( \frac{2ik \sqrt{\sigma_\epsilon(k) + (2-\epsilon)^2}}{\epsilon-2} r  \right),
\end{equation}

\n where

\begin{equation} \label{lambda(k) Whittaker}
\lambda_\epsilon(k) = \frac{i k \sigma_\epsilon(k) (1-\epsilon)}{(2-\epsilon)\sqrt{ \sigma_\epsilon(k) + (2-\epsilon)^2}}~,
\end{equation}

\n and ${\mathcal B}_n$ is given by the same formula except with $W_{\lambda_\epsilon(k),|n|}$ in place of $M_{\lambda_\epsilon(k), |n|}$. The Whittaker functions $M_{\lambda, n}(x)$ and $W_{\lambda, n}(x)$ (for any non-negative integer $n$) are linearly independent solutions of the equation \cite{handbook}

\begin{equation*}
y'' + \left( -\frac{1}{4} + \frac{\lambda}{x} + \frac{\frac{1}{4}-n^2}{x^2} \right) y = 0~.
\end{equation*}

\n Let us take $k_\epsilon = \frac{1}{\epsilon}$ and $\epsilon = \frac{1}{2}$, then

$$\kappa = \sqrt{k_\epsilon^2 - \tfrac{1}{4}} -\frac{i}{2} \approx 1.936 - i0.5$$

\n We show some numerical evidence that $\kappa$ is a limit point of transmission eigenvalues. We conjecture that for each $n=1,2,...$ there exists $k_n \in \mathbb{C} \setminus \{\kappa\}$ such that $f(n,k_n) = 0$ and $k_n \to \kappa$ as $n \to \infty$. In other words, $\kappa$ is a limit point of the transmission eigenvalues $\{k_n\}$.

\n
For each of the values $n=1$, $n=7$, and $n=12$, we present two plots of the functions $\re f(n,x+i \tau)$ and $\im f(n,x+i \tau)$ as functions of $x$, corresponding to two different values of $\tau$. The two values of $\tau$ are chosen to be close to $\im \kappa = -0.5$, and such that they exhibit two different configurations: one for which the intersection point of $\re f$ and $\im f$ is below the horizontal axis, and one for which it is above the horizontal axis. This shows that for some intermediate value of $\tau$ both $\re f$ and $\im f$ vanish. It is reasonable to expect that this common vanishing occurs at a point $x$ near the $x$ values of the two intersection points. One notes that as $n$ increases the $x$ values of the two intersection points get closer to $1.936-i0.5$ Computations for larger values of $n$ were consistent with this.

\begin{figure}[H]
\centering
\captionsetup{width=.84\linewidth}
\includegraphics[scale=0.4]{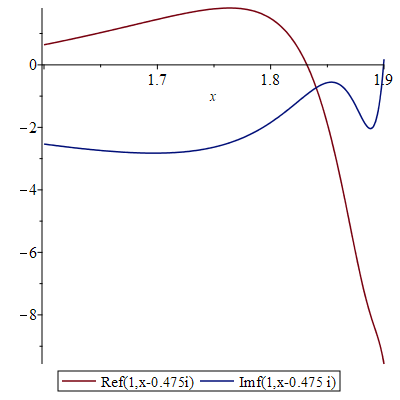}
\qquad
\includegraphics[scale=0.4]{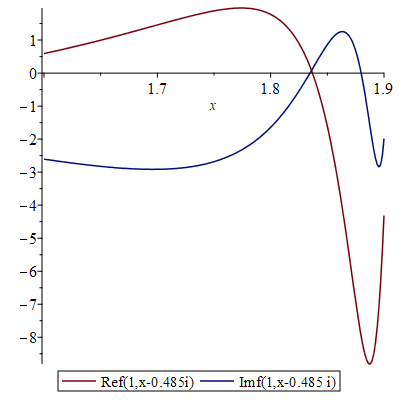}
\caption{Plots of real and imaginary parts of $f(n,x+i\tau)$ for $n=1$ and two different values of $\tau$ indicating where their intersection point crosses the horizontal axis.}
\end{figure}

\begin{figure}[H]
\centering
\captionsetup{width=.84\linewidth}
\includegraphics[scale=0.4]{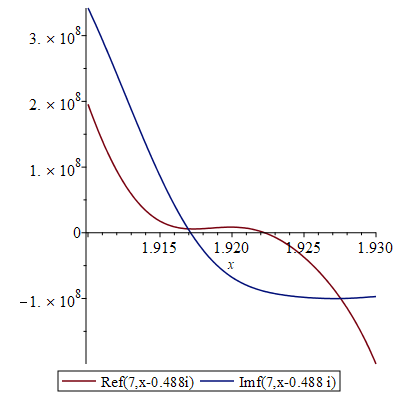}
\qquad
\includegraphics[scale=0.4]{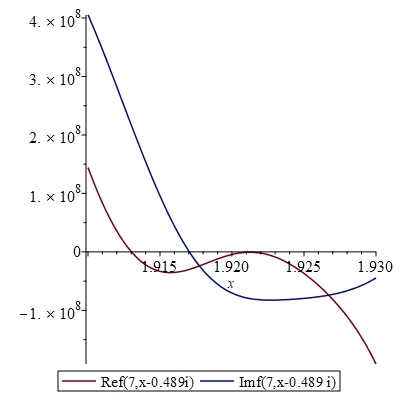}
\caption{Plots of real and imaginary parts of $f(n,x+i\tau)$ for $n=7$ and two different values of $\tau$ indicating where their intersection point crosses the horizontal axis.}
\end{figure}

\begin{figure}[H]
\centering
\captionsetup{width=.84\linewidth}
\includegraphics[scale=0.4]{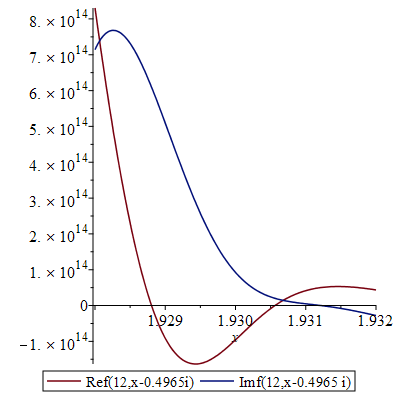}
\qquad
\includegraphics[scale=0.4]{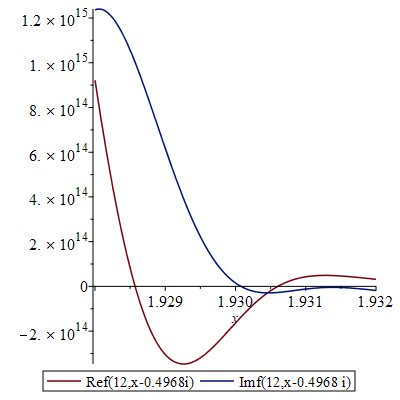}
\caption{Plots of real and imaginary parts of $f(n,x+i\tau)$ for $n=12$ and two different values of $\tau$ indicating where their intersection point crosses the horizontal axis.}
\end{figure}

\section{The scattering estimates}
\setcounter{equation}{0}
In this section we prove Theorems~\ref{THM main} and \ref{THM main coro}. The first observation is that the anisotropy in \eqref{main u_c} can be eliminated, if we change the variables in the transmitted filed $u_c^t$, but leave the incident and scattered fields unchanged. Namely, let

\begin{equation} \label{u^s and u^t}
u^s=u^s_c, \qquad u^t = u_c^t \circ F,
\end{equation}

\n where $F$ is given by \eqref{F}, then $u^t$ is defined in $B_2 \setminus  \overline{B}_\epsilon$.

Invoking Lemma~\ref{LEM change of var} and using the facts that $F = Id$ on $S_2$, $F$ maps $S_\epsilon$ onto $S_1$ and that $F^{-1}_*A_c = I$, we see that \eqref{main u_c} can be equivalently rewritten as

\begin{equation} \label{main u^s u^t}
\begin{cases}
\Delta u^s + k^2 u^s = 0~, \hspace{1in} &\text{in} \quad \RR^d \setminus \overline{B}_2
\\[.05in]
u^s \ \text{satisfies the outgoing radiation condition}
\\[.05in]
\Delta u^t + k^2 q u^t = 0~, &\text{in} \quad B_2 \setminus \overline{B}_\epsilon
\\[.05in]
\Delta u^i + k^2 u^i = 0~, &\text{in} \quad \RR^d
\\[.05in]
u^t = u^i + u^s~,  &\text{on} \quad S_2
\\[.05in]
\partial_{\nu} u^t = \partial_{\nu} u^i + \partial_{\nu} u^s~,  &\text{on} \quad S_2
\\[.05in]
u^t = 0, &\text{on} \quad S_\epsilon~,
\end{cases}
\end{equation}

\n where 

\begin{equation} \label{q}
q(x,k) = F^{-1}_* q_c (x,k) = 
\begin{cases}
1~, \hspace{1.5in} & \text{in} \quad  \RR^d \setminus B_2 
\\[.1in]
\displaystyle 1 + \sigma_\epsilon(k) | \det DF(x)|, & \text{in} \quad  B_2 \setminus B_\epsilon~. 
\end{cases}
\end{equation}

\n Introducing

\begin{equation} \label{u def sc tr inc}
u =
\begin{cases}
u^t~, \qquad \qquad &\text{in} \quad B_2\setminus \overline{B}_\epsilon
\\
u^i + u^s, & \text{in} \quad \RR^d \setminus B_2~,
\end{cases}
\end{equation}

\n the problem \eqref{main u^s u^t} can be rewritten as find $u\in H^1_{loc}(\RR^d \setminus \overline{B}_\epsilon)$

\begin{equation} \label{main u}
\begin{cases}
\Delta u + k^2 q u = 0~, \hspace{1in} &\text{in} \quad \RR^d \setminus \overline{B}_\epsilon
\\[.05in]
\Delta u^i + k^2 u^i = 0~, &\text{in} \quad \RR^d
\\[.05in]
u = 0~, &\text{on} \quad S_\epsilon
\\[.05in]
u-u^i \ \ \text{satisfies the outgoing radiation condition}.
\end{cases}
\end{equation}

\n Here we used that the boundary conditions on $S_2$ from \eqref{main u^s u^t} simply become $\jump{u} = \jump{\partial_\nu u} =0$ on $S_2$, i.e., $u$ and its normal derivative are continuous across $S_2$.

\subsection{The Lippmann-Schwinger equation}

Consider the fundamental solution of the Helmholtz equation in free space: for any $x \neq y$

\begin{equation} \label{fund sol}
\Phi_k(x,y) = 
\begin{cases}
\dfrac{e^{ik|x-y|}}{4\pi |x-y|}~, \qquad \quad &d=3~,
\\[0.15in]
\dfrac{i}{4} H_0^{(1)}(k|x-y|)~, & d=2~.
\end{cases}
\end{equation}

\n We incorporate the homogeneous Dirichlet boundary condition of \eqref{main u} into the fundamental solution, i.e., we let $\Phi_k^0$ be the Green's function for the Helmholtz equation in the region $\RR^d \setminus \overline{B}_\epsilon$ with the Dirichlet boundary condition in $S_\epsilon$. For any fixed $y \in \RR^d \setminus \overline{B}_\epsilon$ $\Phi_k^0(x,y)$ satisfies

\begin{equation} \label{Greens fctn}
\begin{cases}
\Delta_x \Phi^0_k(x,y) + k^2 \Phi^0_k(x,y) = -\delta_y~, \hspace{1in} & \quad x\in \RR^d \setminus \overline{B}_\epsilon
\\[.05in]
\Phi^0_k(x,y) = 0~, & \quad  x\in S_\epsilon
\\[.05in]
\Phi^0_k(\cdot ,y) \ \ \text{satisfies the outgoing  radiation condition}~.
\end{cases}
\end{equation}

\n Clearly we can write

\begin{equation*}
\Phi^0_k(x,y) = \Phi_k(x,y) +  \Psi_k(x,y)~,
\end{equation*}

\n where the function $\Psi_k(\cdot,y)$ is the unique solution to the following exterior Dirichlet boundary value problem for the Helmholtz equation

\begin{equation} \label{Psi}
\begin{cases}
\Delta_x \Psi_k(x, y) + k^2 \Psi_k(x, y) = 0~, \qquad \qquad & x \in \RR^d \setminus \overline{B}_\epsilon
\\
\Psi_k(x, y) = -\Phi_k(x, y)~,  & x \in \ S_\epsilon
\\[.05in]
\Psi_k(\cdot, y) \ \ \text{satisfies the outgoing radiation condition}~.
\end{cases}
\end{equation}

\n Note that the boundary data $-\Phi_k(x, y)$ is smooth, hence the function $\Psi_k(x, y)$ is smooth for $x \in \RR^d \setminus B_\epsilon$ and for any fixed $y$ as above. Next, let us introduce the volume integral operator

\begin{equation} \label{T}
Tu(x) = k^2 \int_{B_2 \setminus \overline{B}_\epsilon} (q(y,k)-1) u(y) \Phi_k^0(x,y) dy.
\end{equation}
Then the solution $u$ of \eqref{main u} satisfies the integral equation 
\begin{equation} \label{LS}
u - T u = u^i + u^{is},
\end{equation}

\n where $u^{is}$ is the scattered field from the ball $B_\epsilon$ due to the incident field $u^i$, i.e., it is the unique solution of

\begin{equation} \label{u^is}
\begin{cases}
\Delta u^{is} + k^2 u^{is} = 0, \qquad \qquad &\text{in} \ \RR^d \setminus \overline{B}_\epsilon
\\
u^{is} = -u^i,  & \text{on} \ S_\epsilon
\\[.05in]
u^{is} \ \ \text{satisfies the outgoing radiation condition}.
\end{cases}
\end{equation}

\n The equation \eqref{LS} is known as the Lippmann-Schwinger equation for the scattering problem \eqref{main u} written in terms of the Green's function $\Phi_k^0$. It can be derived the same way as done for example in \cite{CK19}(without Dirichlet boundary conditions and using the kernel $\Phi_k$). In Lemma~\ref{LEM T bounds} (see also (\ref{3dbound}) and (\ref{2dbound}) ) we prove that for any fixed interval of wave numbers $[k_-,k_+]$, $0<k_-<k_+<\infty$, and any fixed $R>2$ there exists an $\epsilon_0>0$ (depending on $k_+$ and $R$) such that  

\begin{equation} \label{T<1/2}
\|T\|_{L^2(B_R \setminus B_\epsilon) \to L^2(B_R \setminus B_\epsilon)} \leq \frac{1}{2},
\end{equation}
for any $k\in [k_-,k_+]$, and $\epsilon<\epsilon_0$. Therefore  the operator $I-T$ is invertible on $L^2(B_R \setminus B_\epsilon)$ and the integral equation \eqref{LS} has a unique solution $u_R \in L^2(B_R \setminus B_\epsilon)$. Furthermore $u_R= u|_{B_R\setminus B_\epsilon}$ where $u$ is the solution of \eqref{main u}. This follows from the fact that $u|_{B_R\setminus B_\epsilon}$ is in $L^2(B_R\setminus B_\epsilon)$ and as already noted satisfies the integral equation \eqref{LS}. It now follows immediately from \eqref{LS}, and the fact that the domain of integration for the operator $T$ is $B_2\setminus B_\epsilon$, that the solution to \eqref{main u} is given by
$$
u= T u_R + u^i + u^{is}
$$
in all of ${\mathbb R}^d\setminus B_\epsilon$. Note that due to the mapping properties of the volume potential $T u_R$ is in $H^1_{loc}({\mathbb R}^d\setminus B_\epsilon)$. The above argument shows that solving \eqref{main u} is equivalent to solving the Lippmann-Schwinger equation (\ref{LS}) on $B_R \setminus B_\epsilon$ (for any bounded set of wave numbers $[k_-,k_+]$ and $\epsilon$ sufficiently small).

\subsection{Proof of Theorems~\ref{THM main} and \ref{THM main coro}}

The main ingredients of the proofs of Theorems~\ref{THM main} and \ref{THM main coro} are $\epsilon$-explicit estimates for the scattered field $u^{is}$ and the operator $T$ in appropriate Sobolev spaces. We state these estimates in the two lemmata below, however, for clarity of exposition their proofs are postponed to  subsequent sections (see Section~\ref{SECT u^is} and Section~\ref{SECT T}, respectively).

\begin{lemma} \label{LEM T bounds}
Let $T$ be defined by \eqref{T}, and let $M_{\epsilon,k}$ and $a(k)$  be defined by \eqref{M} and \eqref{a(k)}, respectively. Suppose $R>1$, $k_0>0$, and $0< \epsilon k < k_0$. Then for any $u \in L^2(B_2 \setminus B_\epsilon)$

\begin{equation*}
\|T u \|_{L^2(B_R \setminus B_\epsilon)} \lesssim k^2 a(k) M_{\epsilon, k} \|u\|_{L^2(B_2 \setminus B_\epsilon)}~,
\end{equation*}

\n where the implicit constant depends only on $R$ and $k_0$.
\end{lemma}

\begin{lemma} \label{LEM u^is}
Let $u^{is}$ be defined by \eqref{u^is}, let  $R>1$, and $k_0>0$. Assume $0<\epsilon k <k_0$ and that $u^i$ satisfies \eqref{u^i assumption 1}, then

\begin{equation*}
\|u^{is}\|_{L^2(B_R\setminus B_1)} \lesssim
\begin{cases}
\epsilon~, \hspace{0.7in}  &d=3~,
\\[0.1in]
\dfrac{|H_0^{(1)}(k)|}{|H_0^{(1)}(\epsilon k)|}~, & d=2~,
\end{cases}
\end{equation*}

\n and

\begin{equation*}
\|u^{is}\|_{L^2(B_R\setminus B_\epsilon)} \lesssim \epsilon^{d-2}~,
\end{equation*}

\n where the implicit constants depends only on $R, k_0$ and $C$ (the constant from the inequality \eqref{u^i assumption 1} for the incident field $u^i$).
\end{lemma}

\n
With the help of the above lemmata we now prove the following scattering estimate:

\begin{theorem} \label{THM u-u^i bounds}
Let $M_{\epsilon,k}$ and $a(k)$  be defined by \eqref{M} and \eqref{a(k)}, respectively.
Suppose $R>2$, $k_0>0$, and $0<\epsilon k < k_0$, and suppose $u^i$ satisfies \eqref{u^i assumption 1}. Let $u$ be be the solution to \eqref{main u}. There exists a constant $c=c(k_0, R)>0$ such that, if $k^2 a(k) M_{\epsilon, k} < c$, then

\begin{equation} \label{u-u^i upto S_eps}
\|u-u^i\|_{L^2(B_R \setminus B_\epsilon)} \lesssim \epsilon^{d-2} + k^2a(k) M_{\epsilon, k} \|u^i\|_{L^2(B_R)}~, \qquad \qquad for \ d=2,3~~,
\end{equation}

\n and

\begin{equation} \label{u-u^i away S_eps d=2}
\|u-u^i\|_{L^2(B_R \setminus B_1)} \lesssim \frac{|H_0^{(1)}(k)|}{|H_0^{(1)}(\epsilon k)|} + k^2a(k) M_{\epsilon, k} \left( 1 +\|u^i\|_{L^2(B_R)} \right)~, \qquad \qquad for \ d=2~,
\end{equation}

\n where the implicit constants depend only on $R, k_0$ and $C$ (the constant from the inequality \eqref{u^i assumption 1}).
\end{theorem}

\begin{remark}
\normalfont
As an immediate corollary we obtain Theorem~\ref{THM main}, because $u-u^i=u^s=u^s_c$ outside $B_2$. 
\end{remark}

\begin{proof}
Consider the Lippmann-Schwinger equation \eqref{LS} in the space $L^2(B_R \setminus B_\epsilon)$. Lemma~\ref{LEM T bounds} implies that there exists a constant $C_1=C_1(k_0,R)>0$, such that

\begin{equation*}
\|T\|_{L^2(B_R \setminus B_\epsilon) \to L^2(B_R \setminus B_\epsilon)} \leq C_1 k^2 a(k) M_{\epsilon, k}=:r
\end{equation*}

\n Assume that $r <\frac{1}{2}$, or equivalently $k^2 a(k) M_{\epsilon,k} < \frac{1}{2C_1}=:c$. Then the operator $I-T$ is invertible on $L^2(B_R \setminus B_\epsilon)$ and using \eqref{LS} and the Neumann series expansion we obtain

\begin{equation*}
u = (I-T)^{-1} (u^i+u^{is}) = u^i+u^{is} + \sum_{n=1}^\infty T^n (u^i+u^{is}).
\end{equation*}

\n Upon summation of the geometric series, the above equation implies the bound

\begin{equation*}
\begin{split}
\|u-u^i\|_{L^2(B_R \setminus B_\epsilon)} &\leq \|u^{is}\|_{L^2(B_R \setminus B_\epsilon)} + \frac{r}{1-r} \|u^i + u^{is}\|_{L^2(B_R \setminus B_\epsilon)} 
\\
&\leq \|u^{is}\|_{L^2(B_R \setminus B_\epsilon)} + 2r \left( \|u^i\|_{L^2(B_R)} + \|u^{is}\|_{L^2(B_R \setminus B_\epsilon)} \right)
\\
&\lesssim \|u^{is}\|_{L^2(B_R \setminus B_\epsilon)} + r \|u^i\|_{L^2(B_R)} 
\\
&\lesssim \epsilon^{d-2} + r\|u^i\|_{L^2(B_R)},
\end{split}
\end{equation*}

\n where in the last step we used Lemma~\ref{LEM u^is}. This concludes the proof of the inequality \eqref{u-u^i upto S_eps}.

\n To prove \eqref{u-u^i away S_eps d=2}, we take  $d=2$. From the Lippmann-Schwinger equation $u-u^i = u^{is} + Tu$, and hence, using Lemma~\ref{LEM u^is} we have ,

\begin{equation*}
\|u-u^i\|_{L^2(B_R \setminus B_1)} \leq \|u^{is}\|_{L^2(B_R \setminus B_1)} + \|T u\|_{L^2(B_R \setminus B_\epsilon)} \lesssim \frac{|H_0^{(1)}(k)|}{|H_0^{(1)}(\epsilon k)|} + r \|u\|_{L^2(B_R \setminus B_\epsilon)}.
\end{equation*}

\n From \eqref{u-u^i upto S_eps} with $d=2$ we have

\begin{equation*}
\begin{split}
\|u\|_{L^2(B_R \setminus B_\epsilon)} &\leq \|u^i\|_{L^2(B_R)} + \|u-u^i\|_{L^2(B_R \setminus B_\epsilon)} \lesssim \|u^i\|_{L^2(B_R)} + 1 + k^2 a(k) M_{\epsilon, k} \|u^i\|_{L^2(B_R)} 
\\
&\lesssim 1 + \|u^i\|_{L^2(B_R)},
\end{split}
\end{equation*}

\n where in the last step we used the assumption that  $k^2 a(k) M_{\epsilon, k}<c$. A combination of the last two estimates and insertion of  $r= C_1 k^2 a(k) M_{\epsilon, k}$ leads to  \eqref{u-u^i away S_eps d=2}.

\end{proof}

\n Before proceeding to the proof of Theorem~\ref{THM main coro}, we first estimate the far field pattern $u_\infty$, given by \eqref{far field}, in terms of the $L^2(B_5 \setminus B_2)$-norm of the scattered field $u^s$,  given by \eqref{u^s and u^t} and \eqref{main u_c}. 

\medskip
\n
In the following,  by using the term ``an absolute implicit constant", we signify that, the inequality in question holds with a positive constant independent of the involved parameters.
\begin{lemma} \label{LEM far field}
With an absolute implicit constant, for any $|\hat{x}|=1$ and $k>0$,

\begin{equation}
|u_\infty(\hat{x})| \lesssim (1+k^3) \|u^s\|_{L^2(B_5\setminus B_2)} 
\begin{cases}
1~, \qquad &d=3~,
\\[.1in]
k^{-\frac{1}{2}}~, & d=2~.
\end{cases}
\end{equation}

\end{lemma}

\begin{proof}
The far field pattern has the following representation \cite{CK19}:

\begin{equation*}
u_\infty(\hat{x}) = \CI \cdot
\begin{cases}
\dfrac{1}{4\pi}~, \qquad &d=3~,
\\[.1in]
\dfrac{e^{i \frac{\pi}{4}}}{\sqrt{8\pi k}}~, & d=2~,
\end{cases}
\end{equation*}

\n where

$$\CI = \int_{S_4} \left(u^s(y) \partial_{\nu_y} e^{-ik \hat{x} \cdot y} - \partial_\nu u^s(y) e^{-ik \hat{x} \cdot y}\right) ds(y)~,$$
and $S_4$ is the $d-1$-sphere of radius $4$ centered at the origin (note that one could use any $d-1$ manifold circumscribing $B_2$ in its interior).
\n Using H{\"o}lder's inequality and the duality $H^{-\frac{1}{2}} \subset L^2 \subset H^{\frac{1}{2}}$ with the pivot space $L^2$, we can bound

\begin{equation*}
\begin{split}
|\CI|  \lesssim & k \|u^s\|_{L^2(S_4)} + \| e^{-ik \hat{x} \cdot y}\|_{H^{\frac{1}{2}}(S_4)} \|\partial_\nu u^s\|_{H^{-\frac{1}{2}}(S_4)} 
\\[.05in]
\lesssim & k \|u^s\|_{H^1(B_4 \setminus B_3)} + \| e^{-ik \hat{x} \cdot y}\|_{H^1(B_4 \setminus B_3)}  \|\partial_\nu u^s\|_{H^{-\frac{1}{2}}(S_4)} 
\\[.05in]
\lesssim & k \|u^s\|_{H^1(B_4 \setminus B_3)} + (1+k)  \|\partial_\nu u^s\|_{H^{-\frac{1}{2}}(S_4)}~,
\end{split}
\end{equation*}

\n where in the second step we used trace estimates. Next we bound the $H^{-\frac{1}{2}}$-norm of $\partial_\nu u^s$. Given any $\phi \in H^\frac{1}{2}(S_4)$, consider its extension to $B_4 \setminus \overline{B}_3$ via a bounded right inverse of the trace operator: 

\begin{equation} \label{w_phi H^1}
\begin{cases}
w_\phi \in H^1(B_4\setminus \overline{B}_3)
\\
w_\phi = 0~, \qquad &\text{on} \ S_3~,
\\
w_\phi = \phi~, &\text{on} \ S_4~.
\end{cases}
\end{equation}

\n As this defines a bounded operator from $H^\frac{1}{2}(S_3 \cup S_4)$ to $H^1(B_4 \setminus \overline{B}_3)$, we have that with an absolute implicit constant

\begin{equation} \label{w_phi bound}
\|w_\phi\|_{H^1(B_4 \setminus B_3)} \lesssim \|\phi\|_{H^{\frac{1}{2}}(S_4)}~.
\end{equation}

\n Now using the fact that $u^s$ satisfies the Helmholtz equation in $B_4\setminus \overline{B}_3$, we obtain

\begin{equation*}
\langle \partial_\nu u^s, \phi \rangle = \int_{B_4\setminus B_3} \nabla u^s \cdot \nabla w_\phi + w_\phi \Delta u^s dy = \int_{B_4\setminus B_3} \nabla u^s \cdot \nabla w_\phi - k^2 w_\phi u^s dy~,
\end{equation*}

\n where $\langle \cdot, \cdot \rangle$ denotes the dual pairing between $H^{-\frac{1}{2}}(S_4)$ and $H^{\frac{1}{2}}(S_4)$. Using the H{\"o}lder's inequality and \eqref{w_phi bound} we arrive at

\begin{equation*}
|\langle \partial_\nu u^s, \phi \rangle| \lesssim \|\phi\|_{H^{\frac{1}{2}}(S_4)} \left(\|\nabla u^s\|_{L^2(B_4\setminus B_3)} + k^2 \|u^s\|_{L^2(B_4\setminus B_3)} \right)~,
\end{equation*}

\n which readily implies

\begin{equation*}
\|\partial_\nu u^s\|_{H^{-\frac{1}{2}}(S_4)} \lesssim \|\nabla u^s\|_{L^2(B_4\setminus B_3)} + k^2 \|u^s\|_{L^2(B_4\setminus B_3)}~.
\end{equation*}

\n Using that $k+k^2+k^3 \lesssim k+k^3$, we obtain the bound

\begin{equation} \label{I bound}
|\CI| \lesssim (k+k^3) \|u^s\|_{L^2(B_4\setminus B_3)} + (1+k) \|\nabla u^s\|_{L^2(B_4\setminus B_3)}. 
\end{equation}
\n
It remains to bound the $L^2$-norm of $\nabla u^s$, which can be done via the $L^2$-norm of $u^s$ over a larger domain by introducing a cut-off function and using the equation that $u^s$ satisfies. Indeed, let $0\leq \psi \leq 1$ be a cut-off function such that $\text{supp} \ \psi \subset B_5 \setminus \overline{B}_2$, \ $\psi \equiv 1$ on $B_4 \setminus \overline{B}_3$ and $|\nabla \psi| \leq C$ on $B_5 \setminus \overline{B}_2$, with an absolute constant $C>0$. Since

$$\Delta u^s + k^2 u^s = 0, \qquad \qquad \text{in} \ B_5 \setminus \overline{B}_2~,$$

\n multiplication by $\psi^2 \overline{u^s}$ and integration by parts leads to

\begin{equation*}
\begin{split}
\int_{B_5 \setminus B_2} |\nabla u^s|^2 \psi^2 dy &= k^2 \int_{B_5 \setminus B_2} |u^s|^2 \psi^2 dy - 2 \int_{B_5 \setminus B_2} \psi \nabla u^s \cdot \overline{u^s} \nabla \psi dy 
\\
&\leq k^2 \int_{B_5 \setminus B_2} |u^s|^2 \psi^2 dy + \frac{1}{2} \int_{B_5 \setminus B_2} |\nabla u^s|^2 \psi^2 dy + 2 \int_{B_5 \setminus B_2} |u^s|^2 |\nabla \psi|^2 dy~,
\end{split}
\end{equation*}

\n which implies

\begin{equation*}
\frac{1}{2} \int_{B_5 \setminus B_2} |\nabla u^s|^2 \psi^2 dy \leq (k^2 + 2C^2) \int_{B_5 \setminus B_2} |u^s|^2 dy~.
\end{equation*}

\n Consequently,

\begin{equation*}
\|\nabla u^s\|_{L^2(B_4\setminus B_3)} \lesssim (1+k) \|u^s\|_{L^2(B_5\setminus B_2)}~. 
\end{equation*}

\n Combining with \eqref{I bound} we obtain

\begin{equation} \label{I bound 2}
|\CI| \lesssim (k+k^3) \|u^s\|_{L^2(B_5\setminus B_2)} + (1+k)^2 \|u^s\|_{L^2(B_5\setminus B_2)} \lesssim (1+k^3) \|u^s\|_{L^2(B_5\setminus B_2)}~,
\end{equation}
which concludes the proof.
\end{proof}

We are ready to establish the following broadband approximate cloaking estimates:

\begin{theorem}
Let $R>2$ and $k_+>k_->0$, and set $\Gamma = [k_-,k_+]$. Assume that for some constant $c_*>0$, \ $k_\epsilon^2 > c_* \epsilon^{-3}$ for $d=3$, and $k_\epsilon^2 > c_* |\ln \epsilon|/\epsilon$ for $d=2$. Assume further that $u^i$ satisfies the estimate \eqref{u^i assumption 2}. Let $u$ be the solution to \eqref{main u} with $q$ given by \eqref{q}. There exists a constant $c_1=c_1(k_-,k_+,R, c_*)>0$ such that, for all $\epsilon < c_1$ and $k \in \Gamma$

\begin{equation*}
\|u-u^i\|_{L^2(B_R \setminus B_1)} \lesssim
\begin{cases}
\epsilon~, \hspace{1.3in} & d=3~,
\\[.01in]
1/ |\ln \epsilon|~, & d=2~,
\end{cases}
\end{equation*}

\n where the implicit constant depends only on $k_-,k_+,R, c_*$ and $C_R$ (the constants from \eqref{u^i assumption 2}). 
Furthermore, there exists a constant $c_2=c_2(k_-,k_+,c_*)>0$, such that for all $\epsilon < c_2$, \ $k \in \Gamma$ and $|\hat{x}|=1$, 

\begin{equation*}
|u_\infty(\hat{x})| \lesssim
\begin{cases}
\epsilon~, \hspace{1.3in} & d=3~,
\\[.01in]
1/ |\ln \epsilon|~, & d=2~,
\end{cases}
\end{equation*}

\n where the implicit constant depends only on $k_-,k_+, c_*$ and $C_5$.
\end{theorem}

\begin{remark} \mbox{}
\normalfont

\begin{enumerate}
\item[(i)] Since $u-u^i = u^s_c$ outside of $B_2$ Theorem~\ref{THM main coro} follows as an immediate corollary of the above result.
\item[(ii)] For $d=3$ the following proof can be easily modified to show we can bound $u-u^i$ up to the inner boundary $S_\epsilon$, i.e., 

\begin{equation*}
\|u-u^i\|_{L^2(B_R \setminus B_\epsilon)} \lesssim \epsilon~.
\end{equation*}
\end{enumerate}
\end{remark}

\begin{proof}

We note that since $u^i$ is a solution to $\Delta u^i+k^2u^i=0$ in all of ${\mathbb R}^d$, it follows by interior elliptic regularity estimates that for $k\in \Gamma$, $\Vert u^i \Vert_{L^\infty(B_1)}+ \Vert \nabla u^i\Vert_{L^\infty(B_1)} \le C\Vert u^i \Vert_{L^2(B_2)}$, with a constant that only depends on $k_+$. Due to \eqref{u^i assumption 2} we thus conclude that $u^i$, $k\in \Gamma$, satisfies the condition \eqref{u^i assumption 1} as well (for $\epsilon<1$) with a constant that only depends on $C_2$ and $k_+$.

We proceed to estimate $M_{\epsilon, k}$. In view of \eqref{M}, \eqref{q} and Lemma~\ref{LEM F},

\begin{equation*}
M_{\epsilon, k} = |\sigma_\epsilon(k)| \|\det DF\|_{L^\infty(B_2 \setminus B_\epsilon)} = \frac{|\sigma_{\epsilon}(k)|}{(2-\epsilon)^d} \sup_{r \in (\epsilon, 2)} \left( 1+ \frac{2-2\epsilon}{r} \right)^{d-1} = \frac{|\sigma_{\epsilon}(k)|}{(2-\epsilon) \epsilon^{d-1}}.
\end{equation*}

\n Assume that $\epsilon<1$ is so small that

\begin{equation} \label{k_eps ineq}
k_\epsilon^2 \geq \max\left\{ k_+^2, 2(k_+^2-k_+) \right\},
\end{equation}

\n then for any $k\in [0,k_+]$

\begin{equation*}
|\sigma_{\epsilon}(k)| \leq \frac{\sqrt{2}}{|k_\epsilon^2 - k^2| + k} = \frac{\sqrt{2}}{k_\epsilon^2 - k^2 + k},
\end{equation*}

\n where in the last step we used that $k_\epsilon^2 \geq k^2_+$. The function $k\mapsto k_\epsilon^2 - k^2 + k$ is positive and increasing on $[0,\frac12]$, and it is positive and decreasing on $[\frac12,k_+]$ (if $k_+>\frac12$). Thus it follows that 
$$
k_\epsilon^2-k^2+k \ge \min \{ k_\epsilon^2, k_\epsilon^2 - k_+^2+k_+ \} \ge  \frac{k_\epsilon^2}{2}~~ \hbox{ for } k \in [0,k_+]~,
$$
where in the second inequality we have used that $k_\epsilon^2 \geq 2(k_+^2 - k_+)$. As a consequence
\begin{equation*}
\max_{k\in\Gamma} |\sigma_{\epsilon}(k)| \leq \frac{2\sqrt{2}}{k_\epsilon^2} \leq \frac{2\sqrt{2}}{c_*}
\begin{cases}
\epsilon^3, \qquad \qquad &d=3,
\\
\epsilon / |\ln \epsilon|, &d=2.
\end{cases}
\end{equation*}
We now conclude that there exist positive constants $c_0, C_0$ depending only on $c_*$ and $k_+$, such that

\begin{equation} \label{M order eps}
\max_{k\in\Gamma} M_{\epsilon, k} \leq C_0
\begin{cases}
\epsilon, \qquad \qquad &d=3,
\\
1 / |\ln \epsilon|, &d=2,
\end{cases}
\qquad \qquad \forall \ \epsilon \leq c_0~.
\end{equation}

Let us further assume $\epsilon<1/k_+$ so that $0<\epsilon k < 1$ for $k\in \Gamma$. By Theorem~\ref{THM u-u^i bounds} there exists a constant $c=c(R)>0$ such that if $k^2 a(k) M_{\epsilon, k} < c$ (and $k$ is in  $\Gamma$) then

\begin{eqnarray} \label{u-u^i interm}
&&\|u-u^i\|_{L^2(B_R \setminus B_1)} \nonumber \\
&& \hskip 5pt \lesssim 
\begin{cases}
\epsilon + k^2 M_{\epsilon, k} \|u^i\|_{L^2(B_R)} \hspace{0.7in} &\text{for} \ d=3~,
\\[0.05in]
\displaystyle \frac{|H_0^{(1)}(k)|}{|H_0^{(1)}(\epsilon k)|} + k^2 M_{\epsilon, k}  \min\{1+|\ln k|, k^{-\frac{1}{4}}\} \left( 1 + \|u^i\|_{L^2(B_R)}\right) &\text{for} \ d=2~.
\end{cases}
\end{eqnarray}

\n Consider first the case $d=3$. If we assume that $\epsilon < c/C_0 k_+^2$, then

\begin{equation}
\label{3dbound}
\max_{k \in \Gamma} k^2 M_{\epsilon, k} \leq k_+^2 C_0 \epsilon < c~, 
\end{equation}

\n and consequently \eqref{u-u^i interm} can be applied for all $k \in \Gamma$. Using the hypothesis \eqref{u^i assumption 2} and \eqref{M order eps} we conclude that for $\epsilon$ small enough

\begin{equation*}
\max_{k \in \Gamma} \|u-u^i\|_{L^2(B_R \setminus B_1)} \lesssim \epsilon \qquad \qquad \text{for} \ d=3~,
\end{equation*}

\n where the implicit constant depends only on $k_+,R,c_*$ and $C_R$.

Let us now consider $d=2$. The function $a(k) = \min\{1+|\ln k|, k^{-\frac{1}{4}}\}$ is decreasing, therefore, assuming that $\epsilon < e^{-C_0 k_+^2 a(k_-)/c}$ we have

\begin{equation}
\label{2dbound}
\max_{k \in \Gamma} k^2 a(k) M_{\epsilon, k} \leq \frac{C_0}{|\ln \epsilon|} k_+^2 a(k_-) < c~. 
\end{equation}

\n Similarly, as before we conclude that, for $k \in \Gamma$,

\begin{equation*}
\|u-u^i\|_{L^2(B_R \setminus B_1)} \lesssim \frac{|H_0^{(1)}(k)|}{|H_0^{(1)}(\epsilon k)|} + \frac{1}{|\ln \epsilon|}~.
\end{equation*}

\n The function $|H_0^{(1)}(t)|$ is decreasing and $H_0^{(1)}(t) \sim \frac{2}{i\pi} |\ln t|$ as $t \to 0$ (cf. \cite{CC14}). Hence we have the following basic estimates: $|H_0^{(1)}(k)| \leq |H_0^{(1)}(k_-)|$ and

\begin{equation*}
|H_0^{(1)}(t)| \gtrsim |\ln t|, \qquad \qquad \forall \ t \in (0,\tfrac{1}{2}).
\end{equation*}

\n These readily imply the inequality

\begin{equation*}
\|u-u^i\|_{L^2(B_R \setminus B_1)} \lesssim \frac{1}{|\ln (\epsilon k)|} + \frac{1}{|\ln \epsilon|}.
\end{equation*}

\n Since by assumption $\epsilon k_+ < 1$ we get

\begin{equation*}
\min_{k\in \Gamma} |\ln (\epsilon k)| = |\ln (\epsilon k_+)| \geq \frac{1}{2} |\ln \epsilon|~,
\end{equation*}

\n where the last inequality holds, provided $\epsilon < 1/k_+^2$. Putting everything together we conclude that for $\epsilon$ sufficiently small

\begin{equation*}
\max_{k\in\Gamma} \|u-u^i\|_{L^2(B_R \setminus B_1)} \lesssim \frac{1}{|\ln \epsilon|}~.
\end{equation*}

\n
The corresponding estimates for the far field pattern readily follow from Lemma~\ref{LEM far field}.

\end{proof}

\subsection{Scattering from a small obstacle: Proof of Lemma~\ref{LEM u^is}} \label{SECT u^is}
In this section we show that Lemma~\ref{LEM u^is} is a direct consequence of the following result  due to Nguyen and Vogelius \cite{NV12-1} (see also \cite{N10}):

\begin{lemma} \label{LEM NV}
Let $D\subset B_1 \subset \RR^d$ be a smooth open subset with $\RR^d \setminus \overline{D}$ connected. Let $f \in H^{\frac{1}{2}}(\partial D)$, $k_0>0$ and $0<k<k_0$. Let $u$ be the outward radiating solution to the problem

\begin{equation*}
\begin{cases}
\Delta u + k^2 u = 0, \qquad \qquad &\text{in} \ \RR^d \setminus \overline{D}~,
\\
u = f  & \text{on} \ \partial D~.
\end{cases}
\end{equation*}

\n Then for any $\beta \geq 1$

\begin{equation} \label{NV H1 bounds}
\|u\|_{H^1(B_\beta \setminus D)} \lesssim
\begin{cases}
\beta^{\frac{1}{2}} \|f\|_{H^{\frac{1}{2}}(\partial D)}~, \qquad &d=3~,
\\
\beta \|f\|_{H^{\frac{1}{2}}(\partial D)}~, & d=2~,
\end{cases}
\end{equation}

\n where the implicit constant depends only on $k_0$ and $D$ but is independent of $\beta$ and $k$. Furthermore, for $R>1, \beta \geq 1$

\begin{equation} \label{NV L2 Hankel}
\|u\|_{L^2(B_{R\beta} \setminus B_{\beta})} \lesssim \beta \frac{|H_0^{(1)}(\beta k)|}{|H_0^{(1)}(k)|}\|f\|_{H^{\frac{1}{2}}(\partial D)}~, 
\qquad \qquad d=2~,
\end{equation}

\n where the implicit constant depends only on $k_0, D$, and  $R$ but is independent of $\beta$ and $k$.

\end{lemma}

\begin{remark}
\normalfont
The estimate \eqref{NV H1 bounds} for the $L^2$-norm of $u$ and \eqref{NV L2 Hankel}, in the case $R=2$, is proven in Lemma 3 of \cite{NV12-1} under the assumption that $k_0$ is sufficiently small (see also the beginning of the proof of Lemma 4). The subsequent Remark 4 of \cite{NV12-1} explains that these estimates hold without any smallness assumption on $k_0$. The extension of \eqref{NV L2 Hankel} to any $R>1$ is immediate. Finally, the extension from an $L^2$ estimate of $u$ to an $H^1$ estimate, as in \eqref{NV H1 bounds}, is guaranteed by Lemma 4 of \cite{NV12-1}.
\end{remark}

\n As a straightforward consequence of Lemma \ref{LEM NV} (with $D=B_\frac12 \subset B_1$) we obtain the following corollary.
\begin{corollary} (Scattering from a small ball) \mbox{} \label{CORO small ball}

\n Let $\epsilon < \frac{1}{2}$, $R>1, k_0>0$ with $0<2\epsilon k <k_0$. Let $f \in H^{\frac{1}{2}}(S_\epsilon)$ and $u$ be the outward radiating solution of the problem

\begin{equation*}
\begin{cases}
\Delta u + k^2 u = 0, \qquad \qquad &\text{in} \ \RR^d \setminus \overline{B}_\epsilon
\\
u = f & \text{on} \ S_\epsilon
\end{cases}
\end{equation*}

\n Let $R>1$ and set $f_\epsilon = f(2\epsilon \cdot)$, then 

\begin{equation} \label{u small ball}
\|u\|_{L^2(B_R\setminus B_1)} \lesssim \|f_\epsilon\|_{H^\frac{1}{2}(S_\frac12)}
\begin{cases}
\epsilon~, \hspace{0.3in} & d=3~,
\\[0.1in]
\dfrac{|H_0^{(1)}(k)|}{|H_0^{(1)}(\epsilon k)|}~, & d=2~,
\end{cases}
\end{equation}

\n where the implicit constant depends only on $R$ and $k_0$.

\end{corollary}

\begin{remark} \mbox{}
\normalfont

\begin{enumerate}
\item[(i)] The proof of this corollary can be modified in a straightforward way (using only \eqref{NV H1 bounds} of Lemma~\ref{LEM NV}) to yield the following bounds up to the inner boundary $S_\epsilon$:

\begin{equation} \label{u small ball upto S_eps bounds}
\|u\|_{L^2(B_R\setminus B_\epsilon)} \lesssim \epsilon^{d-2} \|f_\epsilon\|_{H^\frac{1}{2}(S_\frac12)}~,
\qquad 
\|\nabla u\|_{L^2(B_R\setminus B_\epsilon)} \lesssim \epsilon^{d-3} (1+k) \|f_\epsilon\|_{H^\frac{1}{2}(S_\frac12)}~,
\end{equation}

\n where again the implicit constants depend only on $R$ and $k_0$. These estimates for $d=3$ are as good as the bound in \eqref{u small ball}, in terms of being of the same order $\epsilon$. However, for $d=2$ the smallness in $\epsilon$ is lost.

\item[(ii)] For scattering estimates in other frequency regimes (e.g. the high frequency case) we refer to \cite{NV12-1}, and also to \cite{HPV07} concerning asymptotically precise estimates for a small circular inhomogeneity and $d=2$.
\end{enumerate}
\end{remark}

\begin{proof}
Let $u_\epsilon(y) = u(2\epsilon y)$, then $u_\epsilon$ is the radiating solution of the problem

\begin{equation*}
\begin{cases}
\Delta u_\epsilon + (2\epsilon)^2 k^2 u_\epsilon = 0, \qquad \qquad &\text{in} \ \RR^d \setminus \overline{B}_\frac12
\\
u_\epsilon = f_\epsilon  & \text{on} \ S_\frac12
\end{cases}
\end{equation*}

\n Let us start with the case $d=2$. By scaling the norm and using the estimate \eqref{NV L2 Hankel} of Lemma~\ref{LEM NV} we obtain

\begin{equation} \label{u small ball d=2 in proof}
\|u\|_{L^2(B_R \setminus B_{1})} = \epsilon \|u_\epsilon\|_{L^2(B_{\frac{R}{2\epsilon}} \setminus B_{\frac{1}{2\epsilon}}) } \lesssim \frac{|H_0^{(1)}(k)|}{|H_0^{(1)}(2\epsilon k)|} \|f_\epsilon\|_{H^\frac{1}{2}(S_\frac12)}. 
\end{equation}

\n It remains to use the estimate

\begin{equation*}
|H_0^{(1)}(\epsilon k)| \lesssim |H_0^{(1)}(2\epsilon k)|,
\end{equation*}

\n which holds true with an absolute implicit constant as the function $H_0^{(1)}$ has no real zeros, and as the functions $H_0^{(1)}(\cdot)$ and $H_0^{(1)}(2\cdot)$ have the same asymptotics at $0$ and at $\infty$.

\n In the case $d=3$ the argument works analogously, giving the bound

\begin{equation*}
\|u\|_{L^2(B_R\setminus B_1)} \lesssim \epsilon \|f_\epsilon\|_{H^\frac{1}{2}(S_\frac12)}.
\end{equation*}
\end{proof}

\n To conclude the proof of Lemma~\ref{LEM u^is} we apply the above corollary and the first estimate in \eqref{u small ball upto S_eps bounds} to the function $u^{is}$. That way we obtain the desired estimates of Lemma~\ref{LEM u^is}, but with the additional factor

$$\|f_\epsilon\|_{H^{\frac{1}{2}}(S_\frac12)} = \|u^i(2\epsilon \cdot)\|_{H^{\frac{1}{2}}(S_\frac12)}$$

\n on the right-hand sides of the inequalities. It thus remains to prove that the above quantity is bounded by a constant depending only on $k_0$. To this end, the standard trace estimate and a rescaling of the norms give

\begin{equation*}
\begin{split}
\|u^i(2\epsilon \cdot)\|_{H^{\frac{1}{2}}(S_\frac12)} &\lesssim \|u^i(2\epsilon \cdot)\|_{H^1(B_\frac12)} = \|u^i(2\epsilon \cdot)\|_{L^2(B_\frac12)} + 2\epsilon \|\nabla u^i(2\epsilon \cdot)\|_{L^2(B_\frac12)} 
\\
&= (2\epsilon)^{-\frac{d}{2}} \|u^i\|_{L^2(B_\epsilon)} + (2\epsilon)^{1-\frac{d}{2}} \|\nabla u^i\|_{L^2(B_\epsilon)} \lesssim \|u^i\|_{L^\infty(B_\epsilon)} + \epsilon \|\nabla u^i\|_{L^\infty(B_\epsilon)} \lesssim 1 ~.
\end{split}
\end{equation*}

\n For the last inequality we used the assumption \eqref{u^i assumption 1}.

\subsection{Bounds for the operator $T$: Proof of Lemma~\ref{LEM T bounds}} \label{SECT T}

Let us split the operator $T$ into two parts: $T = T_1 + T_2$, where

\begin{equation} \label{T1}
T_1u(x) = k^2 \int_{B_2 \setminus \overline{B}_\epsilon} (q(y,k)-1) u(y) \Phi_k(x,y) dy~,
\end{equation}

\n and 

\begin{equation} \label{T2}
T_2u(x) = k^2 \int_{B_2 \setminus \overline{B}_\epsilon} (q(y,k)-1) u(y) \Psi_k(x,y) dy~,
\end{equation}
with $\Psi_k$ given by \eqref{Psi}.  Thus,  to bound $Tu$ on $L^2(B_R \setminus B_\epsilon)$, it suffices to bound $T_1u$ and $T_2u$. We start by deriving some estimates for the fundamental solution, $\Phi_k$, in Lemma~\ref{LEM Phi bounds} below. These are then used in Lemma~\ref{LEM T1 bounds} to obtain bounds for $T_1u$. To bound $T_2u$ we need $L^2(B_\epsilon)$-norm bounds for $T_1u$ and $\nabla T_1u$, with explicit dependence on the small parameter $\epsilon$. Parts $(ii)$ and $(iv)$ of Lemma~\ref{LEM T1 bounds} serve that purpose, and this is where the estimates on the derivatives of the fundamental solution from part $(ii)$ of Lemma~\ref{LEM Phi bounds} will be used. The bound for $T_2u$ is given in Lemma~\ref{LEM T2 bounds}. Finally, Lemma~\ref{LEM T bounds} is a direct consequence of Lemmas~\ref{LEM T1 bounds} and \ref{LEM T2 bounds}. 

\begin{lemma} \label{LEM Phi bounds}

Let $\Phi_k$ be given by \eqref{fund sol}.

\begin{itemize}
\item[(i)] Let $R, r>0$. With implicit constants depending only on $R$ and $r$,

\begin{equation*}
\sup_{x \in B_R} \int_{B_r} |\Phi_k(x,y)|^2 dy \lesssim
\begin{cases}
1~, \hspace{1.3in} & d=3~,
\\
\min\{1+\ln^2 k~, k^{-1}\}, & d=2~.
\end{cases}
\end{equation*}

\item[(ii)] Let $R,r>0$. With an absolute implicit constant (i.e. independent of all the involved parameter $R, r$ and $k$)

\begin{equation*}
\sup_{x \in B_R} \int_{B_r} |\nabla_x \Phi_k(x,y)| dy \lesssim r \left(1+(rk)^{\frac{d-1}{2}}\right).
\end{equation*}

\end{itemize}
\end{lemma}

\begin{proof} \mbox{}
\n
$\bm{\textbf{The case} \ d=3}:$ Let us start by showing that for any $x \in \RR^3$ with implicit constants independent of $x$ and $r$,

\begin{equation} \label{1/|x-y| bounds 3d}
\int_{B_r} \frac{dy}{|x-y|^2}  \lesssim r~,
\qquad \qquad
\int_{B_r} \frac{dy}{|x-y|}  \lesssim r^2~.
\end{equation}

\n We prove only the first inequality, the second follows analogously. Assume first that $x \in B_{2r}$, then $B_r \subset B_{3r}(x)$ and hence

\begin{equation*}
\int_{B_r} \frac{dy}{|x-y|^2}  \leq \int_{B_{3r}(x)} \frac{dy}{|x-y|^2}  = \int_{B_{3r}} \frac{dz}{|z|^2}  = 4\pi \int_0^{3r} d\rho =12\pi r~.
\end{equation*}

\n If now $x \notin B_{2r}$ we use that $|y-x|\geq |x|-|y| \geq 2r-r = r$ for any $y \in B_r$, so that

\begin{equation*}
\int_{B_r} \frac{dy}{|x-y|^2}  \leq \frac{1}{r^2} |B_r| \lesssim r~.
\end{equation*}

\n Consequently, we immediately obtain

\begin{equation*}
\int_{B_r} |\Phi_k(x,y)|^2 dy \lesssim \int_{B_r} \frac{dy}{|x-y|^2}  \lesssim r~.
\end{equation*}

\n This concludes the proof of part $(i)$. Let us turn to gradient bounds. Direct calculation shows that

\begin{equation*}
\nabla_x \Phi_k(x,y) = \tfrac{1}{4\pi} e^{ik|x-y|} \left( ik|x-y|-1 \right) \frac{x-y}{|x-y|^3}~,
\end{equation*}

\n and hence

\begin{equation} \label{grad Phi bound 3d}
|\nabla_x \Phi_k(x,y)| = \frac{\sqrt{1+k^2 |x-y|^2}}{4\pi |x-y|^2} \leq \frac{1}{4\pi |x-y|^2} + \frac{k}{4\pi |x-y|}~.
\end{equation}

\n From \eqref{1/|x-y| bounds 3d} we conclude that

\begin{equation*}
\int_{B_r} |\nabla_x \Phi_k(x,y)| dy \lesssim r + r^2 k~.
\end{equation*}

$\bm{\textbf{The case} \ d=2}: $ Analogously to \eqref{1/|x-y| bounds 3d}, for any $x \in \RR^2$ with implicit constants independent of $x$ and $r$,

\begin{equation} \label{1/|x-y| bounds 2d}
\int_{B_r} \frac{dy}{|x-y|}  \lesssim r,
\qquad \qquad
\int_{B_r} \frac{dy}{\sqrt{|x-y|}}  \lesssim r^\frac{3}{2}.
\end{equation}

\n We use the asymptotic relations \cite{CC14}

\begin{equation*}
H_0^{(1)}(t) \sim \frac{2}{i\pi} |\ln t|~, \quad \text{as} \ t \to 0~,
\qquad \text{and} \qquad
H_0^{(1)}(t) \sim \sqrt{\frac{2 \pi}{t}} e^{i(t-\frac{\pi}{4})}~, \quad \text{as} \ t \to \infty~,
\end{equation*}

\n to obtain the bound

\begin{equation*}
|H_0^{(1)}(t)| \lesssim |\ln t| \chi_{(0,\frac12)}(t) + \frac{1}{\sqrt{t}} \chi_{(\frac{1}{2},\infty)}(t)~, \qquad \qquad \forall t \geq 0~,
\end{equation*}

\n which then implies

\begin{equation*}
\begin{split}
\int_{B_r} |\Phi_k(x,y)|^2 dy &\lesssim \int_{B_r} \left[\ln^2 (k|x-y|) \chi_{(0,\frac12)}(k|x-y|) + \frac{1}{k|x-y|} \chi_{(\frac{1}{2},\infty)}(k|x-y|)\right] dy 
\\[.05in]
&= \int_{B_r \cap B_{\frac{1}{2k}}(x)} \ln^2 (k|x-y|) dy + \int_{B_r\cap B^C_{\frac{1}{2k}}(x)} \frac{1}{k|x-y|} dy =:I_1 + I_2~.
\end{split}
\end{equation*}

\n Let us start by bounding $I_2$. Using that $|x-y|>\frac{1}{2k}$ it is clear that $I_2 \lesssim 1$ with implicit constant depending only on $r$. This bound can be improved when $k$ is large. Indeed, to get a better bound in that case, observe that

\begin{equation*}
\sup_{x \in B_R} I_2 \leq \frac{1}{k} \sup_{x \in B_R} \int_{B_r} \frac{1}{|x-y|} dy \lesssim \frac{1}{k}~,
\end{equation*}

\n where the last inequality follows from \eqref{1/|x-y| bounds 2d}. Combining, the two estimates, we have (with an implicit constant depending only on $r$)

\begin{equation*}
\sup_{x \in B_R} I_2 \lesssim \min\{1, k^{-1}\}~.
\end{equation*}

\n Let us turn to bounding $I_1$. Dropping $B_r$ from the integration and changing the variables $z=y-x$ inside the integral, we get

\begin{equation}
\label{I-1interm1}
I_1 \leq \int_{B_{\frac{1}{2k}}} \ln^2 (k|z|) dz = \frac{1}{k^2} \int_{B_\frac12} \ln^2 (|z|) dz \lesssim \frac{1}{k^2}~,
\end{equation}

\n where in the last step we used that $\ln^2|z|$ has an integrable singularity at $z=0$. This bound can be improved when $k$ is small. Dropping $B_{\frac{1}{2k}}(x)$ from the integral $I_1$ and using the inequality

\begin{equation*}
\ln^2 (k|x-y|) \lesssim \ln^2 k + \ln^2 |x-y|,
\end{equation*}

\n we arrive at the estimate

\begin{equation}
\label{I-1interm2}
\sup_{x\in B_R} I_1 \lesssim \ln^2 k + \sup_{x \in B_R} \int_{B_r} \ln^2 |x-y| dy \lesssim \ln^2 k + 1~.
\end{equation}

\n The last inequality is easily established, based on the estimate

\begin{equation*}
\ln^2 t \lesssim \frac{1}{t} \chi_{(0,1)}(t) + t \chi_{1,\infty)}(t), \qquad \qquad \forall t \geq 0~.
\end{equation*}

\n Indeed,

\begin{equation*}
\begin{split}
\sup_{x \in B_R} \int_{B_r} \ln^2 |x-y| dy &\lesssim \sup_{x \in B_R} \int_{B_r \cap B_1(x)} \frac{1}{|x-y|} dy + \sup_{x \in B_R} \int_{B_r \cap B_{1}^C(x)} |x-y| dy 
\\
&\leq  \int_{B_1} \frac{1}{|z|} dz + (r+R) |B_r| \lesssim 1~.
\end{split}
\end{equation*}

\n Combining the two estimates \eqref{I-1interm1} and \eqref{I-1interm2}, we arrive at 

\begin{equation*}
\sup_{x\in B_R} I_1 \lesssim \min\{1 + \ln^2 k, k^{-2} \}~.
\end{equation*}

\n Finally, a combination of the bounds for $I_1$ and $I_2$ yields that

\begin{equation*}
\sup_{x\in B_R} \int_{B_r} |\Phi_k(x,y)|^2 dy \lesssim \min\{1 + \ln^2 k, k^{-2} \} + \min\{1,k^{-1}\} \lesssim \min\{1+\ln^2 k, k^{-1}\}~.
\end{equation*}

\n For gradient bounds in 2d we use the asymptotic relations

\begin{equation*}
{H_0^{(1)}}'(t) \sim -\frac{2}{i\pi t}~, \quad \text{as} \ t \to 0~,
\qquad \text{and} \qquad
{H_0^{(1)}}'(t) \sim i \sqrt{\frac{2 \pi}{t}} e^{i(t-\frac{\pi}{4})}~, \quad \text{as} \ t \to \infty~,
\end{equation*}

\n along with the bound

\begin{equation*}
|{H_0^{(1)}}'(t)| \lesssim \frac{1}{t} \chi_{(0,1)}(t) + \frac{1}{\sqrt{t}} \chi_{(1,\infty)}(t)~, \qquad \qquad \forall t \geq 0~.
\end{equation*}

\n Since

\begin{equation*}
\nabla_x \Phi_k(x,y) = \frac{i k}{4} {H_0^{(1)}}'(k|x-y|) \frac{x-y}{|x-y|},  
\end{equation*}

\n we obtain, with the help of \eqref{1/|x-y| bounds 2d}, that 

\begin{equation*}
\begin{split}
\int_{B_r} |\nabla_x \Phi_k(x,y)| dy &\lesssim \int_{B_r \cap B_{\frac{1}{k}}(x)} \frac{dy}{|x-y|} + \sqrt{k} \int_{B_r \cap B^C_{\frac{1}{k}}(x)} \frac{dy}{\sqrt{|x-y|}}  
\\
&\lesssim \min\{r,k^{-1}\} + \sqrt{k} r^{\frac{3}{2}} 
\\[.05in]
&= r \left[  \min\{1,(rk)^{-1}\} + \sqrt{rk} \right] \leq r \left(1+\sqrt{rk} \right)~,
\end{split}
\end{equation*}
with an implicit constant independent of $r$ and $x$.
\end{proof}

\begin{lemma} \label{LEM T1 bounds}
Let $T_1$ be defined by \eqref{T1}, $R>1$, \ $\epsilon<1$ and $M_{\epsilon, k}$ be given by \eqref{M}. Then for any $u \in L^2(B_2 \setminus B_\epsilon)$, 

\begin{itemize}
\item[(i)] $\displaystyle \|T_1u\|_{L^2(B_R \setminus B_\epsilon)} \lesssim k^2 M_{\epsilon, k} \|u\|_{L^2(B_2 \setminus B_\epsilon)} 
\begin{cases}
1~, \hspace{1.3in} & d=3~,
\\
\min\{1+|\ln k|, k^{-\frac{1}{2}}\}~, & d=2~.
\end{cases}$

\item[(ii)] $\displaystyle \|T_1u\|_{L^2(B_\epsilon)} \lesssim \epsilon^{\frac{d}{2}} k^2 M_{\epsilon, k} \|u\|_{L^2(B_2 \setminus B_\epsilon)}
\begin{cases}
1~, \hspace{1.3in} & d=3~,
\\
\min\{1+|\ln k|, k^{-\frac{1}{2}}\}~, & d=2~.
\end{cases}$

\item[(iii)] $\displaystyle \|\nabla T_1u\|_{L^2(B_R \setminus B_\epsilon)} \lesssim k^2 (1+k^{\frac{d-1}{2}}) M_{\epsilon, k} \|u\|_{L^2(B_2 \setminus B_\epsilon)}$

\item[(iv)] $\displaystyle \|\nabla T_1u\|_{L^2(B_\epsilon)} \lesssim \sqrt{\epsilon} k^2 (1+k^{\frac{d-1}{4}}) (1+(\epsilon k)^{\frac{d-1}{4}}) M_{\epsilon, k} \|u\|_{L^2(B_2 \setminus B_\epsilon)}$

\end{itemize}

\n where all the implicit constants are independent of $u, \epsilon$ and $k$. The implicit constants in $(i)$ and $(iii)$ depend only on $R$; and those in $(ii)$ and $(iv)$ are absolute constants.
\end{lemma}

\begin{proof}

\n We set $\Omega_R = B_R \setminus B_\epsilon$, in particular  $\Omega_2 = B_2 \setminus B_\epsilon$. Note that, by H{\"o}lder's inequality we have

\begin{equation*}
|T_1u(x)| \leq k^2 M_{\epsilon, k} \|u\|_{L^2(\Omega_2)} \|\Phi_k(x,\cdot)\|_{L^2(\Omega_2)},
\end{equation*}

\n which implies the estimate

\begin{equation} \label{T_1u L^2 bound via Phi}
\|T_1u\|_{L^2(\Omega_R)} \leq k^2 M_{\epsilon, k} \|u\|_{L^2(\Omega_2)} \left(\int_{\Omega_R} \|\Phi_k(x,\cdot)\|_{L^2(\Omega_2)}^2 dx \right)^{\frac{1}{2}}.
\end{equation}

\n Using part $(i)$ of Lemma~\ref{LEM Phi bounds}, we obtain that 

\begin{equation*}
\int_{\Omega_R} \int_{\Omega_2} |\Phi_k(x,y)|^2 dy dx \leq |B_R| \sup_{x \in B_R} \int_{B_2} |\Phi_k(x,y)|^2 dy \lesssim
\begin{cases}
1~, \hspace{1.3in} & d=3~,
\\
\min\{1+\ln^2 k, k^{-1}\}~, & d=2~,
\end{cases}
\end{equation*}

\n where the implicit constant depends only on $R$. This concludes the proof of part $(i)$. 

\n The proof of $(ii)$ proceeds analogously, with $B_\epsilon$ in place of $\Omega_R$, and the conclusion follows from the estimate

\begin{equation*}
\begin{split}
\int_{B_\epsilon} \|\Phi_k(x,\cdot)\|_{L^2(\Omega_2)}^2 dx &\leq |B_\epsilon| \sup_{x \in B_\epsilon} \int_{B_2 \setminus B_\epsilon} |\Phi_k(x,y)|^2 dy \lesssim \epsilon^d \sup_{x \in B_1} \int_{B_2} |\Phi_k(x,y)|^2 dy  
\\
&\lesssim \epsilon^d 
\begin{cases}
1~, \hspace{1.3in} & d=3~,
\\[.05in]
\min\{1+\ln^2 k, k^{-1}\}~, & d=2~.
\end{cases}
\end{split}
\end{equation*}

The above direct estimation argument cannot be used to bound the $L^2$-norm of $\nabla T_1 u$, as $\nabla T_1 u$ is an integral operator whose kernel is not square integrable. However, we can obtain bounds using interpolation. To this end,  differentiating inside the integral we have

$$\nabla T_1u(x) = \int_{\Omega_2} K(x,y) u(y) dy=:T_1^gu(x)~, \qquad \quad K(x,y) = k^2 (q(y,k)-1) \nabla_x \Phi_k(x,y)~.$$

\n Clearly,

\begin{equation*}
\begin{split}
\|T_1^g u\|_{L^\infty(\Omega_R)} &\leq \sup_{x \in \Omega_R} \int_{\Omega_2} |K(x,y)| dy \cdot \|u\|_{L^\infty(\Omega_2)} \leq k^2 M_{\epsilon, k} \|u\|_{L^\infty(\Omega_2)} \sup_{x \in \Omega_R} \int_{\Omega_2} |\nabla_x \Phi_k(x,y)| dy~,
\\
\|T_1^g u\|_{L^1(\Omega_R)} &\leq \sup_{y \in \Omega_2} \int_{\Omega_R} |K(x,y)| dx \cdot \|u\|_{L^1(\Omega_2)} \leq k^2 M_{\epsilon, k} \|u\|_{L^1(\Omega_2)} \sup_{y \in \Omega_2} \int_{\Omega_R} |\nabla_x \Phi_k(x,y)| dx~.
\end{split}
\end{equation*}

\n Using part $(ii)$ of Lemma~\ref{LEM Phi bounds}, we get

\begin{equation*}
\sup_{x \in \Omega_R} \int_{\Omega_2} |\nabla_x \Phi_k(x,y)| dy \leq \sup_{x \in B_R} \int_{B_2} |\nabla_x \Phi_k(x,y)| dy \lesssim 1 + k^{\frac{d-1}{2}}~,
\end{equation*}

\n and noting that $\nabla_x \Phi_k(x,y) = - \nabla_y \Phi_k(y,x)$, we similarly get 

\begin{equation*}
\sup_{y \in \Omega_2} \int_{\Omega_R} |\nabla_x \Phi_k(x,y)| dx \leq \sup_{y \in B_2} \int_{B_R} |\nabla_y \Phi_k(y,x)| dx \lesssim 1 + k^{\frac{d-1}{2}}~,
\end{equation*}

\n where the implicit constants depend only on $R$. Thus we obtain that $T_1^g: L^1(\Omega_2) \to L^1(\Omega_R)$ and $T_1^g:L^\infty(\Omega_2) \to L^\infty(\Omega_R)$ both have operator norms bounded by $C k^2(1 + k^{\frac{d-1}{2}}) M_{\epsilon, k}$, where $C$ is a constant depending only on $R$. The Marcinkiewicz interpolation theorem \cite{GRAF} now implies that $T_1$ maps $L^2(\Omega_2)$ into $L^2(\Omega_R)$ with the operator norm bound

\begin{equation*}
\|T_1^g\|_{L^2(\Omega_2) \to L^2(\Omega_R)} \leq 2\sqrt{2} C k ^2(1 + k^{\frac{d-1}{2}}) M_{\epsilon, k}~, 
\end{equation*}

\n which concludes the proof of part $(iii)$.

\n The proof of part $(iv)$ proceeds analogously, with $B_\epsilon$ in place of $\Omega_R$. Part $(ii)$ of Lemma~\ref{LEM Phi bounds} implies the estimate 

\begin{equation*}
\sup_{y \in \Omega_2} \int_{B_\epsilon} |\nabla_x \Phi_k(x,y)| dx \lesssim \epsilon (1 + (\epsilon k)^{\frac{d-1}{2}}),
\end{equation*}

\n with an absolute implicit constant. We then conclude that

\begin{equation*}
\|T_1^g\|_{L^\infty(\Omega_2) \to L^\infty(B_\epsilon)} \leq C k^2(1+k^{\frac{d-1}{2}}) M_{\epsilon, k} 
\quad \text{and} \quad
\|T_1^g\|_{L^1(\Omega_2) \to L^1(B_\epsilon)} \leq C k^2 \epsilon (1 + (\epsilon k)^{\frac{d-1}{2}}) M_{\epsilon, k},
\end{equation*}

\n where $C$ is an absolute constant. Again using the Marcinkiewicz interpolation theorem we obtain

\begin{equation*}
\|T_1^g\|_{L^2(\Omega_2) \to L^2(B_\epsilon)} \leq 2 \sqrt{2} C k^2 M_{\epsilon, k} \left[ \epsilon (1+k^{\frac{d-1}{2}}) (1+(\epsilon k)^{\frac{d-1}{2}}) \right]^{\frac{1}{2}}.
\end{equation*}

\end{proof}

\begin{lemma} \label{LEM T2 bounds}
Let $T_2$ be defined by \eqref{T2}, $k_0>0$ and $R>1$. Suppose $0< \epsilon k < k_0$ and let $M_{\epsilon,k}$ be given by \eqref{M}. Then for any $u \in L^2(B_2 \setminus B_\epsilon)$

\begin{equation}
\|T_2 u \|_{L^2(B_R \setminus B_\epsilon)} \lesssim k^2 M_{\epsilon, k} \|u\|_{L^2(B_2 \setminus B_\epsilon)}
\begin{cases}
\sqrt{\epsilon}, \hspace{1.3in} & d=3,
\\
\min\{1+|\ln k|, k^{-\frac{1}{4}}\}, & d=2.
\end{cases}
\end{equation}

\n where the implicit constant depends only on $R$ and $k_0$.
\end{lemma}

\begin{proof}

\n Let $v= T_2 u$ and $f = -T_1u$, then using that $Tu$ solves the problem \eqref{Tu pde}, we conclude that $v$ is the outward radiating solution to the problem

\begin{equation*}
\begin{cases}
\Delta v + k^2 v = 0, \hspace{0.5in} &\text{in} \quad \RR^d \setminus \overline{B}_\epsilon
\\[.05in]
v = f, &\text{on} \quad S_\epsilon.
\end{cases}
\end{equation*}

\n As before we introduce the notation $f_\epsilon(x) = f(2 \epsilon x)$. Using Corollary \ref{CORO small ball} (and the remark following) specifically the first estimate of (\ref{u small ball upto S_eps bounds}) we now get, for $d=3$,

\begin{equation*}
\begin{split}
\|v\|_{L^2(B_R \setminus B_\epsilon)} \lesssim \epsilon \|f_\epsilon\|_{H^{\frac{1}{2}}(S_\frac12)} &\lesssim \epsilon \| T_1u(2\epsilon \cdot)\|_{H^1(B_\frac12)}
\\
&\lesssim \frac{1}{\sqrt{\epsilon}} \|T_1u\|_{L^2(B_\epsilon)} + \sqrt{\epsilon} \|\nabla T_1u\|_{L^2(B_\epsilon)} 
\\[.01in]
&\lesssim \epsilon k^2 M_{\epsilon,k} (1+\sqrt{k}) \|u\|_{L^2(B_2 \setminus B_\epsilon)}~,
\end{split}
\end{equation*}

\n where in the last step we used the parts $(ii)$ and $(iv)$ of Lemma~\ref{LEM T1 bounds}. To conclude the proof, it remains to observe that $\epsilon (1+ \sqrt{k}) \lesssim \sqrt{\epsilon}$, due to the bound $\epsilon k <k_0$.

\n Similarly, for $d=2$ we have

\begin{equation*}
\begin{split}
\|v\|_{L^2(B_R \setminus B_\epsilon)} &\lesssim \|f_\epsilon\|_{H^{\frac{1}{2}}(S_\frac12)} \lesssim \| T_1u(2\epsilon \cdot)\|_{H^1(B_\frac12)}  
\\
&\lesssim \frac{1}{\epsilon} \|T_1u\|_{L^2(B_\epsilon)} + \|\nabla T_1u\|_{L^2(B_\epsilon)}
\\[.01in]
&\lesssim k^2 M_{\epsilon,k} C_{\epsilon, k} \|u\|_{L^2(B_2 \setminus B_\epsilon)}~,
\end{split}
\end{equation*}

\n where

$$C_{\epsilon, k} = \min\{1+|\ln k|, k^{-\frac{1}{2}}\} + \sqrt{\epsilon} (1+k^{\frac{1}{4}})~.$$

\n Using that $\epsilon<1$ and $\epsilon k \leq k_0$ we have $\sqrt{\epsilon} \lesssim \min\{1,k^{-\frac{1}{2}}\}$, which then implies

$$C_{\epsilon, k} \lesssim \min\{1+|\ln k|, k^{-\frac{1}{2}}\} + (1+k^{\frac{1}{4}}) \min\{1,k^{-\frac{1}{2}}\} \lesssim \min\{1+|\ln k|, k^{-\frac{1}{4}}\}$$

\n This completes the proof of Lemma \ref{LEM T2 bounds}.
\end{proof}
\section*{Appendix}
\subsection*{Discreteness of the transmission eigenvalues}

Here we prove part $(iii)$ of Theorem~\ref{THM transmission}. The proof is based on the lemma below.

\begin{lemma} \label{LEM discrete}
Assume $\sqrt{2} k_\epsilon > 1$, let $R>0$ be large, and $h, h_0 >0$ be small enough, such that, with the notation $k=a+ib$, the following sets are nonempty

\begin{equation*}
\begin{split}
\CR_{R,h} &= \{k \in \mathbb{C}: |k|<R\} \cap \left( \{k: a,b>0\} \cup \left\{k: a > \max\{|b|, \sqrt{b^2+b+k_\epsilon^2}\} + h\right\} \right)
\\[.1in]
\CL_{R,h, h_0} &= \{k \in \mathbb{C}: |k|<R\} \cap \left( \{k: a<0, b<-\tfrac{1}{2}-h_0\} \cup \left\{k: a < - \max\{|b|, \sqrt{b^2+b+k_\epsilon^2}\} - h\right\} \right)
\\[.1in]
\CU_{R,h} &= \{k \in \mathbb{C}: |k|<R\} \cap \left( \left\{k: b>|a|+h\right\}  \cup \left\{k: b < -|a|-h \ \text{and} \ b > -\tfrac{1}{2} \left(k_\epsilon^2 + \tfrac{1}{2}\right) \right\} \right)
\end{split}
\end{equation*}

\n Then the interior transmission eigenvalues of \eqref{trans change of var} that lie inside $\CR_{R,h} \cup \CL_{R,h, h_0} \cup \CU_{R,h}$ form a discrete set in (i.e., an at most countable set  with no limit points in $\CR_{R,h} \cup \CL_{R,h, h_0} \cup \CU_{R,h}$).
\end{lemma}

\vspace{.1in}

\n To see that this lemma concludes the proof of Theorem~\ref{THM transmission} consider the following unions:

\begin{equation*} 
\CR = \bigcup_{R=1}^\infty \bigcup_{n=1}^\infty \CR_{R, \frac{1}{n}},
\qquad \qquad
\CL = \bigcup_{R=1}^\infty \bigcup_{n=1}^\infty \bigcup_{m=1}^\infty \CL_{R, \frac{1}{n}, \frac{1}{m}},
\qquad \qquad
\CU = \bigcup_{R=1}^\infty \bigcup_{n=1}^\infty \CU_{R, \frac{1}{n}},
\end{equation*}

\n Lemma \ref{LEM discrete} guarantees the discreteness of the interior transmission eigenvalues inside the union of these sets. We note that

\begin{equation*}
\begin{split}
\CR &= \{k: a,b>0\} \cup \left\{k: a > \max(|b|, \sqrt{b^2+b+k_\epsilon^2}) \right\}
\\
\CL &= \{k: a<0, b<-\tfrac{1}{2}\} \cup \left\{k: a < - \max(|b|, \sqrt{b^2+b+k_\epsilon^2}) \right\}
\\
\CU &= \left\{k: |b| > |a| \ \text{and} \ b > -\tfrac{1}{2} \left(k_\epsilon^2 + \tfrac{1}{2}\right) \right\}. 
\end{split}
\end{equation*}

\n Finally, the symmetry of the set of interior transmission eigenvalues implies that the discreteness also holds in $-\overline{\CR}$ and $-\overline{\CL}$, where the bar denotes complex conjugation. Since $\CG\subset \CR\cup \CL \cup \CU \cup -\overline{\CR} \cup -\overline{\CL}\cup \RR\cup i\RR$, a combination of these discreteness results and (i) of Theorem \ref{THM transmission} yields the proof of the last assertion in Theorem \ref{THM transmission} .

\begin{proof}[Proof of Lemma~\ref{LEM discrete}]
We start by showing the discreteness in the sets $\CR_{R,h}, \CL_{R,h, h_0}$, which are open, connected and disjoint. Let $\lambda = \lambda_1 + i \lambda_2 \in \mathbb{C}$ with $\lambda_1, \lambda_2>0$ to be chosen later. Consider the bounded sesquilinear forms on $X$ (cf. \eqref{X}) given by

\begin{equation*}
\begin{split}
\CA_k(u,\varphi) &= \int_{\mathcal O} \frac{1}{1-q} \left( \Delta u + k^2u \right) \left( \Delta \overline{\varphi} + k^2 \overline{\varphi} \right) dx + k^2  \int_{B_2} \nabla u \cdot \nabla \overline{\varphi} dx + \lambda \int_{B_2} u \overline{\varphi} dx,
\\[.05in]
\CB_k(u,\varphi) &= -(k^4+\lambda) \int_{B_2} u \overline{\varphi} dx
\end{split}
\end{equation*}

\n In terms of $\CA_k$ and $\CB_k$, the variational form of the interior transmission eigenvalue problem \eqref{variational form} reads: $\CA_k(u,\varphi)+\CB_k(u,\varphi)=0$ for all $\varphi \in X$. Since $\CB_k$ yields a compact operator, the discreteness of these eigenvalues, in the regions where both $\CA_k$ and $\CB_k$ depend analytically on $k$, will follow from the  Analytic Fredholm Theory, as in \cite[Section 8.5]{CK19}, once we prove that $\lambda = \lambda(R,h,h_0)$ can be chosen such that $\CA_k$ becomes coercive \cite{CCH12}. For shorthand let us introduce the notation

\begin{equation*}
\frac{1}{1-q(x,k)} = \gamma_k p(x)~, \qquad \quad \gamma_k = k^2+ik-k_\epsilon^2 \quad \text{and} \quad p(x) = \frac{1}{\det DF(x)}~.
\end{equation*}

\n Let $m, M$ be such that

$$0 < m \leq p(x) \leq M, \qquad \qquad \forall x \in B_2 \setminus \overline{B_\epsilon}={\mathcal O}~.$$

\n We consider cases:

\vspace{.1in}
$\bullet$ Let $k = a + i b$ be such that $|k|<R$ and (I) $a>0, \ b > 0$ or (II) $a<0, \ b< -\tfrac{1}{2} - h_0$.

\n Then

\begin{eqnarray} \label{A_k(u,u) for complex k}
\CA_k (u, u) &&= \gamma_k \int_ {\mathcal O} p |\Delta u|^2 dx  + k^2 \|\nabla u\|^2_{B_2} + 2\gamma_k k^2 \int_ {\mathcal O} p \re (\overline{u} \Delta u) dx \nonumber \\
&&\hskip 40pt+ \gamma_k k^4 \int_ {\mathcal O} p |u|^2 dx + \lambda \|u\|^2_{B_2}~,   
\end{eqnarray}

\n where we use the notation $\|u\|_\Omega = \|u\|_{L^2(\Omega)}$. In both cases (I) and (II) we see that

\begin{equation*}
\im \gamma_k = a (2b + 1)>0, \quad \text{and} \quad \im (k^2) = 2ab>0~.
\end{equation*}

\n Hence

\begin{equation*}
\begin{split}
|\im \CA_k (u, u)| \geq &  \im \gamma_k \int_{\mathcal O} p |\Delta u|^2 dx  + \im (k^2) \|\nabla u\|^2_{B_2} + \lambda_2 \|u\|^2_{B_2} -
\\
&-2|\im (\gamma_k k^2)| \cdot \left| \int_ {\mathcal O} p \re (\overline{u} \Delta u) dx \right| - |\im (\gamma_k k^4)| \int_ {\mathcal O} p |u|^2 dx~.
\end{split}
\end{equation*}

\n Let us use the lower bound $p \geq m$ in the first integral. In the integral of the term $2 \re (\overline{u} \Delta u )$ we use H{\"o}lder's inequality along with the estimate $p\leq M$, and then apply Cauchy's inequality with $\delta>0$ to the resulting term. The result becomes

\begin{equation*}
\begin{split}
|\im \CA_k (u, u)| \geq& (m \im \gamma_k - 2M |\im (\gamma_k k^2)| \delta) \|\Delta u\|^2_{\mathcal O} + \im (k^2) \|\nabla u\|^2_{B_2} +
\\
&+\left( \lambda_2 - M|\im (\gamma_k k^4)| - \tfrac{M |\im (\gamma_k k^2)|}{2\delta} \right) \|u\|^2_{\mathcal O}
\end{split}
\end{equation*}

\n If $\im (\gamma_k k^2) = 0$, then coercivity follows for any $\lambda_2 > \max\{M|\im(\gamma_k k^4)|: |k|<R\}$. Otherwise, let us choose $\delta$ such that $4M |\im (\gamma_k k^2)| \delta = m \im \gamma_k$. Then the first term of above inequality is positive and the third term will be positive if

\begin{equation*} 
\lambda_2 > \sup \left\{ M|\im (\gamma_k k^4)| + \frac{2M^2 \left[ \im (\gamma_k k^2) \right]^2}{m \im \gamma_k} : |k|<R \ \text{and} \ \ \text{(I)  or  (II)  holds} \right\}~.
\end{equation*}

\n It remains to see that the above supremum is finite. The first term inside the supremum is bounded and establishing the boundedness of the second term amounts to showing that

\begin{equation*}
\frac{\left[ \im (\gamma_k k^2) \right]^2}{\im \gamma_k} = \frac{a \left[ 4a^2 b - 4b^3 -2b k_\epsilon^2 + a^2 - 3b^2 \right]^2}{2b+1}
\end{equation*}
is bounded. Clearly, in the case (I) this is bounded with a constant depending only on $R$ and $k_\epsilon$, and in the case (II) it is bounded with a constant depending on $R,k_\epsilon$ and $h_0$. Finally, the coercivity follows upon applying Poincare's inequality as $X \subset H^1_0(B_2)$.

\vspace{.1in}
$\bullet$ Let $k = a + i b$ be such that $|k|<R$ and 

\begin{equation} \label{|a| > max}
|a| > \max \left\{|b|, \sqrt{b^2 + b + k_\epsilon^2}\right\} + h~.
\end{equation}

\n Then

\begin{equation*}
\re \gamma_k = a^2 - (b^2 + b + k_\epsilon^2)>0, \quad \text{and} \quad \re (k^2) = a^2 - b^2>0~.
\end{equation*}

\n Repeating the argument of the previous case, only taking real parts in \eqref{A_k(u,u) for complex k}, we obtain that coercivity follows after choosing

\begin{equation*}
\lambda_1 > \sup \left\{ M|\re (\gamma_k k^4)| + \frac{2M^2 \left[ \re (\gamma_k k^2) \right]^2}{m \re \gamma_k} : |k|<R \ \ \text{and} \ \ \eqref{|a| > max} \ \text{holds} \right\}~.
\end{equation*}

\n Clearly this supremum is finite as $\re \gamma_k > h^2$.

It remains to prove the discreteness in the set $\CU_{R,h}$. This can be done analogously, only now $\lambda$ in the sesquilinear forms must be chosen to be a real and negative number with very large absolute value. Coercivity then follows by deriving a lower bound on $|\re {\CA}_k (u, u)|$.

\end{proof}

\section*{Acknowledgments} 
The research of F. Cakoni was partially supported by the AFOSR Grant FA9550-20-1-0024 and  NSF Grant DMS-2106255. The research of MSV was partially supported by NSF Grant DMS-22-05912.  

\bibliographystyle{plain}
\bibliography{ref}
\end{document}